\RequirePackage{hyperref}
\documentclass[11pt, reqno]{amsart}
\usepackage{eqnarray,amsmath,mathtools}
\usepackage{amsthm}

\usepackage{amsfonts, amssymb, stmaryrd}
\usepackage{amsbsy}
\hypersetup{colorlinks,linkcolor={blue!50!black},citecolor={red!50!black},urlcolor={blue!80!black}}
\pdfoutput=1
\usepackage{graphicx} 
\newcommand\rotpi{\rotatebox[origin=c]{180}{$\Pi$}}
\usepackage[all]{xy}

\DeclareMathAlphabet{\mathpzc}{OT1}{pzc}{m}{it}
\usepackage{bbm,dsfont}
\newcommand{\one}{[0]}
\newcommand{\two}{[1]}
\usepackage[paperwidth=8.5in,paperheight=11in,text={6.in,8.5in},centering]{geometry}
\usepackage{bm}

\usepackage{color} 

\newtheorem{lemma}[equation]{Lemma}
\newtheorem{theorem}[equation]{Theorem}

\newtheorem{proposition}[equation]{Proposition}
\newtheorem{corollary}[equation]{Corollary}
\newtheorem*{corollary*}{Corollary}

\theoremstyle{definition}
\newtheorem{definition}[equation]{Definition}

\newtheorem{remark}[equation]{Remark}
\newtheorem{example}[equation]{Example}
\newtheorem{notation}[equation]{Notation}

\numberwithin{equation}{section}

\newcommand{\fm}{\mathfrak m}

\newcommand{\frs}{\mathfrak s}
\newcommand{\frt}{\mathfrak t}
\newcommand{\fru}{\mathfrak u}

\newcommand{\cC}{{\mathcal C}}

\newcommand{\cH}{{\mathcal H}}

\newcommand{\cO}{{\mathcal O}}

\newcommand{\BB}{\mathbb{B}}
\newcommand{\N}{\mathbb{N}}
\newcommand{\R}{\mathbb{R}}

\newcommand{\Z}{\mathbb{Z}}
\newcommand{\D}{\mathbb{D}}
\newcommand{\DD}{\mathbb{D}}

\newcommand{\UU}{\mathbb{U}}

\usepackage{mathrsfs}
\newcommand{\scrP}{\mathscr{P}}
\newcommand{\scrX}{\mathfrak{X}}
\newcommand{\scrI}{\mathscr{I}}
\newcommand{\scrT}{\mathscr{T}}

 \newcommand{\scC}{\mathscr{C}}
 \newcommand{\scD}{\mathscr{D}}

\newcommand{\Crl}{\mathsf{Crl}}
\newcommand{\Vect}{\mathsf{Vect}}
\newcommand{\RelVect}{\mathsf{RelVect}}
\newcommand{\RelSet}{\mathsf{RelSet}}
\newcommand{\Set} {\mathsf{Set}}

\newcommand{\FinSet}{\mathsf{FinSet}}
\newcommand{\sfC} {\mathsf{C}}
\newcommand{\sfD} {\mathsf{D}}
\newcommand{\Man} {\mathsf{Man}}
\newcommand{\RelMan} {\mathsf{RelMan}}
\newcommand{\HyDS} {\mathsf{HyDS}}

\newcommand{\Cat} {\mathsf{Cat}}
\newcommand{\CAT} {\mathsf{CAT}}

\newcommand{\Graph}{\mathsf{Graph}}
\newcommand{\DS}{\mathsf{DS}}
\newcommand{\OS}{\mathsf{OS}}
\newcommand{\HyOS}{\mathsf{HyOS}}
\newcommand{\SSub}{\mathsf{SSub}}

\newcommand{\HyPh}{\mathsf{HyPh}}
\newcommand{\HySSub}{\mathsf{HySSub}}
\newcommand{\HDS}{\mathsf{HyDS}}
\newcommand{\Free}{\mathsf{Free}}

\newcommand{\on}{\mathrm{on}}
\newcommand{\off}{\mathrm{off}}

\newcommand{\st}{\mathsf{st}}
\newcommand{\tot}{\mathsf{tot}}
\newcommand{\op}{\mathsf{op}}
\newcommand{\inter}{\mathsf{int}}
\newcommand{\pr}{\mathsf{pr}}

\newcommand{\id}{\mathsf{id}}

\newcommand{\sF}[1]{\mathsf{#1}}

\newcommand{\inv}{^{-1}}
\newcommand{\toto}{\rightrightarrows}

\DeclareMathOperator{\colim}{colim}

\usepackage{longtable}

\usepackage{tikz}
\usepackage{tikz-cd}
\usetikzlibrary{arrows,cd,decorations.markings,
  calc,chains,matrix,positioning,scopes,snakes}
\tikzset{
   tick/.style={postaction={
      decorate,
      decoration={markings, mark=at position 0.5 with {\draw[-] (0,.4ex) -- (0,-.4ex);}}}
   }
}

\title{Networks   of hybrid open 
  systems}

\author{Eugene Lerman and James Schmidt}
\address{Department of Mathematics, University of Illinois at Urbana-Champaign,
1409 W. Green Street
Urbana, IL 61801, USA}

\begin{document}

\begin{abstract}
We generalize the results of \cite{L2} to the setting of hybrid
systems.  In particular, we introduce the notions of hybrid open
systems, their networks and maps between networks.   
A network of systems is a blueprint for building a larger system out
of smaller subsystems by specifying a pattern of interactions between
subsystems --- an interconnection map.  Maps between networks allow us
to produce maps between complex hybrid dynamical systems by specifying maps
between their subsystems.
\end{abstract}

\maketitle

\tableofcontents

\section{Introduction}
In this paper we generalize the results of \cite{L2} to the setting of hybrid
systems.  In particular, we introduce the notions of hybrid open
systems, their networks and maps between networks. 
A network of systems is a blueprint for building a larger system out
of smaller subsystems by specifying a pattern of interactions between
subsystems --- an interconnection map.  Maps between networks allow us
to produce maps between complex hybrid dynamical systems by specifying maps
between their subsystems.   The framework of \cite{L2} was developed to connect two
rather different views of networks of continuous time systems ---
``networks are morphisms in colored operads'' of Spivak and
collaborators \cite{Spivak, VSL} on one hand and the coupled cell network
formalism of Golubitsky, Stewart and their collaborators (see
\cite{Golubitsky.Stewart.06} and references therein) and its
subsequent generalizations \cite{DL1, DL2}.  By generalizing the
results of \cite{L2} we bring the operadic point of view to hybrid
systems and, at the same time, generalize coupled cell network
formalism to hybrid dynamical systems.

To carry out our program we need to develop an appropriate framework.
The first step is to introduce the notion of a {\em hybrid phase
  space} (Definition~\ref{def:hyph}) and its underlying manifold with
corners.   In this way a hybrid dynamical system is a pair $(a, X)$
where $a$ is a hybrid phase space and $X$ is a vector field on the
underlying manifold $\UU(a)$.   We then define maps of hybrid phase
spaces thereby  making hybrid phase spaces and their maps into a
category $\HyPh$.  We show the category of hybrid phase spaces has
finite products and that the assignment of the underlying manifold to
a hybrid phase space extends to a product-preserving functor
\[
\UU: \HyPh \to \Man
  \]
where $\Man$ denotes the category of manifolds with corners (see
Appendix~\ref{app:A}).

The construction of the category $\HyPh$ and the underling manifold
functor $\UU$ allows us construct the category $\HyDS$ of hybrid
dynamical systems as a ``category of elements'' for a functor with
values in the category $\RelVect$ of vector spaces and linear
relations (see Section~\ref{sec:cat_of_el}).  We provide evidence that
the machine we have built so far makes sense by observing that
executions of hybrid dynamical systems {\em are} maps of hybrid
dynamical systems and therefore are morphisms in the category $\HyDS$
(Definition~\ref{def:exec} and Remark~\ref{rmrk:exec}).  We prove that
maps of hybrid dynamical systems send executions to executions
(Theorem~\ref{thm:exec}) thereby providing another sanity check on our
theory building.

We single out a class of maps between hybrid phase spaces that we call
{\sf hybrid surjective submersions} by
requiring that the corresponding maps of underlying manifolds are
surjective submersions.  We organize hybrid surjective submersions
into a category $\HySSub$ (Definition~\ref{def:hybssub}).  Our
motivation for introducing this category is the following.

The underlying manifold functor $\UU$ extends to a functor
$\UU:\HySSub\to \SSub$ where $\SSub$ is the category of surjective
submersions of manifolds with corners (Notation~\ref{not:ssub}).
Recall that for every surjective submersion
$a= a_\tot \xrightarrow{p_a} a_\st$ ($a_\st$ is the space of states,
$a_\tot$ is the space of states {\em and} controls, and $p_a$ is the
surjective submersion) there corresponds a vector space $\Crl(a)$ of
all control systems on $a$ (Definition~\ref{def:open_sys} and
Notation~\ref{not:crl}).  The assignment $a\mapsto \Crl(a)$ extends to
a $\RelVect$-valued functor $\Crl$.  We therefore have the category of
elements $\OS: = \int \Crl$ of open (control) systems (see
Section~\ref{sec:cat_of_el}).  The category of elements of the
composite functor $\Crl\circ \UU: \HySSub \to \RelVect$ is a category
of hybrid open systems $\HyOS$.  In particular, a {\sf hybrid open
  system} is a pair $(a, F)$ where $a= a_\tot\xrightarrow{p_a} a_\st$
is a hybrid surjective submersion and
$F:\UU(a_\tot) \xrightarrow{\UU(p_a)} T\UU(a_\st)$ is a control
system.  While hybrid open systems are certainly known, we believe
that our definition of the category $\HyOS$ is new.  For other
category-theoretic approaches to hybrid open systems see \cite{Ames},
\cite{TPL} and \cite{LS}.

As was observed in \cite{L2} it is useful to organize surjective
submersions into a double category $\SSub^\Box$ (double categories are
reviewed in Subsection~\ref{subsec:double}).  The second kind of
1-arrows in $\SSub^\Box$ are interconnection morphisms: see
Definition~\ref{def:ssub}, the subsequent remark and
Example~\ref{ex:intercon}.  Roughly speaking these morphisms
describe the effect on open systems of plugging state variables into
controls.  The functor $\Crl:\SSub \to \RelVect$ extends to a functor
of double categories $\Crl:\SSub ^\Box \to \RelVect^\Box$ where
$\RelVect^\Box$ is the double category of vector spaces, linear maps
and linear relations.   We then use the forgetful functor $\UU:\HySSub
\to \SSub$ to turn hybrid surjective submersions into a double
category  $\HySSub ^\Box$ and to define interconnection maps of hybrid
surjective submersions.  As a sanity check we show that a hybrid
dynamical system whose underlying hybrid phase space is a product of
two hybrid phase spaces can be obtained by interconnecting two hybrid
open systems  (see
Example~\ref{ex:hds_on_product}). 

At this point we are almost done with building the machinery.  We
define a network of hybrid open systems to be a pair
$(\{a_x\}_{x\in X}, \psi:b\to \prod _{x\in X} a_x)$ where $\{a_x\}_{x\in X}$
is a collection of hybrid surjective submersions indexed by a finite
set $X$ and $\psi:b\to \prod
_{x\in X}$ is an interconnection map
(Definition~\ref{def:network_hy_os}).   We then define maps of
networks of hybrid open systems
(Definition~\ref{def:map_of_networks_hy_os}), which parallels the
definition of maps of networks of continuous time open systems in
\cite{L2}.   The definition consists of a list of compatible data.  In
more detail a map from a network $(\{a_x\}_{x\in X}, \psi:b\to \prod
_{x\in X} a_x)$ to a network $(\{d_y\}_{y\in Y}, \nu:c\to \prod _{y\in
  Y} d_y)$ consists of a map of finite sets $\varphi:X\to Y$, a
collection $\{\Phi_x: d_{\varphi(x)} \to a_x\}_{x\in X}$ of maps of
hybrid surjective submersions and another map $f:c\to b$ of hybrid
surjective submersions which is compatible with $\varphi$, $\Phi =
\{\Phi_x\}_{x\in X}$, $\psi$ and $\nu$ in an appropriate sense.

The main theorem of this paper (Theorem~\ref{thm:main_result}) can be
roughly phrased as follows.   Recall that
the functor $\Crl:\SSub \to \RelVect$ assigns to a map 
$h:p\to q$ of surjective submersions  a linear relation
$\Crl (h)\subset \Crl (p) \times \Crl(q)$.

Let $\big((\varphi,\Phi),f\big):
  (\{a_x\}_{x\in X}, \psi:b\to \prod _{x\in X} a_x) \to
  (\{d_y\}_{y\in Y}, \nu:c\to \prod _{y\in
  Y} d_y)$ be a map of networks of hybrid open systems.  
Theorem~\ref{thm:main_result} shows that for any choice of
collections of control systems $\{w_x\in
\Crl(\UU(a_x))\}_{x\in X}$ and $\{u_y\in \Crl (\UU(d_y))\}_{y\in Y}$ 
so that $u_{\varphi(x)}$ is
$\UU(\Phi_x)$-related to $w_x$ for all $x\in X$ we get a map
\[
f: (c, \nu^*(\prod u_y)) \to (b,
\psi^* (\prod_{x\in X}w_x))
  \]
of hybrid open systems.  Here $\nu^*: \Crl (\UU(\prod_{y\in Y} d_y)) \to \Crl
(\UU(c))$ and $\psi^*: \Crl (\UU(\prod_{x\in X} a_x)) \to \Crl (\UU
(b))$ are linear maps induced by the interconnection maps $\nu$ and
$\psi$ respectively.   In other words, compatible patterns of
interconnection of hybrid open systems give rise to maps of hybrid open
systems.
  In the case where the two interconnections result in closed
  systems,  the map  $f$ is a map of hybrid dynamical systems.

  Theorem~\ref{thm:main_result}  is a direct generalization of
  \cite[Theorem~9.5]{L2} from continuous time systems to hybrid
  systems.   If one views a network as a directed graph with nodes
  decorated with phase spaces, then it is reasonable to view a map
  between networks as a map of graphs of some sort that preserves the
  decorations.  This is the point of view taken in the papers on
  coupled cell networks cited above.    But one can also view a
  network as a finite list of objects $\{a_1,\ldots, a_k\}$ in some
  monoidal category $\cC$ together with an arrow $a_1\otimes\cdots
  \otimes a_k \to b$ in $\cC$.   In other words, a network is a
  morphism in the colored operad $\cO\cC$ associated with $\cC$.   This
  is a view of networks advocated by Spivak and his collaborators ({\em op.\ cit.}), and
  several other applied category theorists.  If one believes
that networks should form  a category, then the operadic approach to
networks does not quite fit since there are no  arrows between
morphisms in colored operads.  A solution to this issue was proposed
in \cite{L2} in the special case of continuous time systems (there $\cC ^\mathrm{op}$ is
 a subcategory $\SSub_{\textrm{int}}$ of the Cartesian monoidal
 category $\SSub$ of
surjective submersions).  In particular, maps between networks of
continuous time systems were shown to give rise to maps of continuous
time systems (open and closed).  A sanity check was provided by
proving that directed graphs with nodes decorated with manifolds give
rise to morphism in the operad $\cO(
{\SSub_{\textrm{int}}}^\mathrm{op})$ of surjective submersions.
Fibrations of decorated graphs were shown to give rise to maps between
networks of continuous time systems and to maps of dynamical systems.
The ``higher operad'' of surjective submersions (i.e., the colored
operad $\cO({\SSub_{\textrm{int}}}^\mathrm{op})$ together with the morphisms between
networks) was in turn interpreted by constructing  a map to the double
category $\RelVect^\Box$ of vector spaces, linear maps and linear
relations.

In the present paper we develop an analogous theory for hybrid
dynamical systems.  Our  motivation, to some extent, is  to have another natural example
of  a ``higher operad'' before we proceed with building a general
theory of ``higher operads'' suitable for constructing categories of
networks.   Since we don't have such a general theory yet, we choose
not to emphasize the operadic aspects of
Theorem~\ref{thm:main_result}.

It is easy to specialize Theorem~\ref{thm:main_result} to the subcategory
 of networks of hybrid systems constructed from decorated graphs
(the nodes now are decorated with hybrid phase spaces).  Analogous to
a result in \cite{L2}, fibrations of decorated graphs 
give rise to maps between hybrid systems.  In particular, this 
allow an extension of the coupled cell network formalism to hybrid systems.
We plan to address this elsewhere.

Finally, we note that our  hybrid dynamical systems  are not necessarily
deterministic.  This is done for two reasons.  First of all,
non-deterministic hybrid systems have a wide range of applications.
Secondly, the formalism is simpler.  With a bit more work the approach
developed in this paper can handle deterministic systems:
 see \cite{Schmidt}.

  \subsection*{Related work} We will not attempt to survey all the
  work done on hybrid dynamical systems.  The area is vast.  Most of
  the papers in the area appear in engineering publications.  As far
  as we know the use of category theory in hybrid systems is rather
  limited.  We note the work of Ames in his thesis \cite{Ames} and in
  a number of subsequent publications.  Other notable papers are by
  Tabuada, Pappas and Lima \cite{TPL} and by Lorenco and
  Sernadas \cite{LS}.  A radically different approach to hybrid open
  systems and their interconnection has recently been developed by
  Schultz and Spivak \cite{SS}.  It is based on sheaves.  It would be
  interesting to understand how the formalism of \cite{SS} relates to
  the traditional view of hybrid systems and to the formalism
  developed here.\\

\noindent {\bf Acknowledgments:} E. L.\ thanks Sayan Mitra for many hours of
conversations. E. L.\ also thanks Michael Warren for some useful discussions.

\section{Background}
\subsection{Relations and linear relations}\mbox{}\\
We start by setting our notation for relations and their compositions.
We view a   relation $X\xrightarrow{R} Y$ from a set $X$ to a set
$Y$ as a generalization of the notion of a function from $X$ to $Y$.
It also generalizes the notion of a {\em partial} function from $X$ to
$Y$.  The following  definitions are standard.
\begin{definition} \label{notation:rel}
We define a {\em relation} $R:X\to Y$ from a set $X$ to a set $Y$ to be 
a subset of the product  $X\times Y$.  The composition
of relations $(Z\xleftarrow{S}Y) \circ
(Y\xleftarrow{R}X)$ is defined by
\begin{equation}\label{eq2.1}
S\circ R := \{ (x,z) \in X\times Z \mid \textrm{ there exists } y\in Y
\textrm{ with } (y,z)\in S, (x,y)\in R\}.
\end{equation}
A relation $X\xrightarrow{R}Y$ is a {\sf function} if for any $x\in X$
the intersection $ (\{x\}\times Y) \cap R$ is a singleton.   In other
words we identify a function $X\xrightarrow{f}Y$ with its graph
$\Graph(f) := \{ (x,y) \in X\times Y \mid y = f(x)\}$.  A relation
$X\xrightarrow{R}Y$ is a {\sf partial function } if for any $x\in X$
the intersection  $( \{x\} \times Y) \cap R$ is either empty or a single
point (and so defines a function from a subset of $X$ to $Y$).
\end{definition}

\begin{definition}[The 2-category $\RelSet$]
Sets and relations form a 2-category $\RelSet$. The objects of
$\RelSet$ are sets.  The 1-arrows of $\RelSet$ are relations with composition
defined by \eqref{eq2.1}.  A 2-arrow from a
relation $X\xrightarrow{R}Y$ to a relation $X\xrightarrow{R'}Y$ is an
inclusion $R\hookrightarrow R'$.
\end{definition}

\begin{definition}[The 2-category $\RelVect$]  \label{def:relvect}
  Vector spaces and linear relations form a 2-category $\RelVect$. The
  objects of $\RelVect$ are vector spaces.  A 1-arrow
  $W\xrightarrow{R} V$ in $\RelVect$ is a
  linear subspace $R$ of $W\times V$, that is, a linear relation.  The
  composition of linear relations is defined by \eqref{eq2.1}, that is,
  it is defined exactly the same way as the composition of
  set-theoretic relations.  A 2-arrow from a linear relation
  $W\xrightarrow{R}V$ to a linear relation $W\xrightarrow{R'}V$ is a
  (linear) inclusion $R\hookrightarrow R'$.
\end{definition}
The following definition is standard for morphisms in $\RelSet$.  It
works equally well in $\RelVect$.
\begin{definition} \label{def:transpose}
  Given a linear relation $R:V\to W$ its {\sf transpose} is the
  relation $R^T:W\to V$ defined by
  \[
R^T =\{ (w, v) \in W\times V\mid (v,w) \in R\}.
   \] 
 \end{definition}

\subsection{Double categories} \label{subsec:double}\mbox{}

 The 2-category $\RelVect$ is the horizontal 2-category of a double
 category $\RelVect^\Box$ defined below.  To define $\RelVect^\Box$ we
 need to recall the notion of a double category which is due to
 Charles Ehresmann. These are categories which may be defined as
 categories internal to the category $\mathsf{CAT}$ of categories
 \cite{BMM, Shul, Shulman10} (we notationally distinguish between the
 2-category $\CAT$ of not necessarily small categories and the
 2-category $\Cat$ of small categories).  Double categories, double 
 functors and vertical transformations will  get a fair amount of use in this paper.
 
\begin{definition}[Double category]
\label{def:double_cat}
 A {\sf double category} $\D$ consists of two categories $\D_1$
(of arrows) and $\D_0$ (of objects) together with four structure
{\em functors}:
\begin{gather*}
  \frs, \frt:\D_1 \to \D_0 \quad (\textrm{source and target}), \qquad 
  \fru: \D_0 \to \D_1\quad (\textrm{unit})\\
  \fm: \D_1\times_{\frs, \D_0, \frt} \D_1\to \D_1 , 
\quad  (\xLeftarrow{\alpha}, \xLeftarrow{\beta}) \mapsto 
\fm(\xLeftarrow{\alpha}, \xLeftarrow{\beta})=: \alpha* \beta
\qquad (\textrm{multiplication/composition})
\end{gather*}
so that
\begin{gather*}
\frs \circ \fru = \id_{\D_0}= \frt \circ \fru, \\
\frs (\alpha *\beta) = \frs (\beta),\qquad \frt (\alpha *\beta) = \frt (\alpha),\\
(\alpha * \beta)* \gamma =
\alpha *(\beta *\gamma),\\
\alpha* \fru(\frs(\alpha)) = \alpha\quad \textrm{and} \quad\fru(\frt(\alpha))* \alpha =\alpha
\end{gather*}
for all arrows $\alpha, \beta, \gamma$ of  $\D_1$.
\end{definition}
\begin{remark}
 There are weaker notions
of double categories such as pseudo-double categories.  We won't
use them in this paper.  The
reader should be warned that pseudo-double categories are often
referred to as double categories.
\end{remark}

\begin{notation} Let $\D$ be a double category. We call the objects of
  the category $\D_0$ {\sf $0$-cells} or {\sf
    objects} and the morphisms of $\D_0$ the ``vertical'' {\sf
    1-morphisms}. We call the 
  objects of the category $\D_1$ ``horizontal'' {\sf 1-cells}  (or the
  ``horizontal'' {\sf 1-morphisms}).  A morphism $\alpha:\mu\Rightarrow \nu$ of
  $\D_1$ with $\frs (\alpha) = (a\xrightarrow{f} b)$ and $\frt(\alpha)
  =(c\xrightarrow{g} d)$ is a {\sf
    2-morphism} or a {\sf 2-cell} (we will use the two terms
  interchangeably).  We may depict such a  2-cell   as
  \begin{equation}\label{eq:2-mor}
    \xy
(-8, 6)*+{c} ="1"; 
(8, 6)*+{a} ="2";
(-8,-6)*+{d }="3";
(8, -6)*+{b}="4";
{\ar@{->}_{\mu} "2";"1"};
{\ar@{->}^{\nu} "4";"3"};
{\ar@{->}_{g} "1";"3"};
{\ar@{->}^{f} "2";"4"};
{\ar@{=>}^{\alpha} (0,3)*{};(0,-3)};
\endxy
\end{equation}
with the arrows of $\D_0$ drawn vertically, the objects of $\D_1$
drawn as horizontal arrows and the 2-cell $\alpha$ drawn as a double
arrow from $\mu$ to $\nu$.  Note that
\eqref{eq:2-mor} is {\bf not} necessarily a 2-commuting diagram in
some 2-category.
\end{notation}

\begin{remark}
Equivalently one can define a (strict) double category $\D$ as consisting of
``tiles'' or ``squares'' that can be composed by either stacking the
squares  vertically or horizontally.  The vertical composition of
squares correspond to the composition of morphisms in $\D_1$.  The horizontal
composition of squares corresponds to  the functor $\fm$.  Given 4
composable 
squares 
\[
  \xy
  (-12, 12)*+{\bullet} ="1";
  (0,12 )*+{\bullet} ="2";
  (12,12 )*+{\bullet} ="3";
  (-12,0 )*+{\bullet} ="4";
  (0,0 )*+{\bullet} ="5";
  (12, 0)*+{\bullet} ="6";
  (-12,-12 )*+{\bullet} ="7";
  (0,-12 )*+{\bullet} ="8";
  (12,-12 )*+{\bullet} ="9";
  {\ar@{<-} "1";"2"};
  {\ar@{<-} "2";"3"};
  {\ar@{<-} "4";"5"};
  {\ar@{<-} "5";"6"};
  {\ar@{->} "4";"7"};
  {\ar@{->}"5";"8"};
   {\ar@{->}"6";"9"};
  {\ar@{->}"1";"4"};
  {\ar@{->} "2";"5"};
  {\ar@{->} "3";"6"};
   {\ar@{->} "8";"7"};
  {\ar@{->} "9";"8"};
   \endxy
 \]
 we can first compose them vertically in pairs and then compose them
 horizontally or the other way around.  Since $\fm:\D_1\times_{\D_0}
 \D_1\to \D_1$ is a functor, the results are equal.
 
 This suggests that a double category $\D$ may also be viewed as a
 category in $\CAT$ whose category of objects $\D_0'$ has the objects
 of $\D_1$ as arrows (with the composition defined by the functor
 $\fm$).  The category of arrows $\D_1'$ has the arrows of $\D_0$ as
 objects (and the squares of $\D$ as morphisms now composed
 horizontally).  That is, we may take the directed tiles of the double
 category $\D$ and reflect them along the appropriate diagonal.  Since our
 horizontal arrows point left and vertical arrows point down, c.f.\ \eqref{eq:2-mor}, the diagonal in
 question is
 the southwest -- northeast
 diagonal.  For this reason ``horizontal'' and ``vertical''
 terminology of double categories is somewhat arbitrary: presenting a
 double category as a category internal to $\CAT$ hides this
 reflection symmetry and introduces a bias.
 \end{remark} 

 \begin{definition}[Horizontal 2-category of a double category]
   Associated to a double category $\D$ there is a {\sf horizontal
   2-category $\cH(\D)$} whose objects are the objects of $\D$,
   1-arrows are the horizontal morphisms of $\D$ and 2-arrows are the
   2-cells of $\D$ of the form 
\begin{equation} \label{eq:globular}
\xy
(-8, 6)*+{c} ="1"; 
(8, 6)*+{a} ="2";
(-8,-8)*+{c }="3";
(8, -8)*+{a}="4";
{\ar@{->}_{\mu} "2";"1"};
{\ar@{->}^{\nu} "4";"3"};
{\ar@{->}_{\id} "1";"3"};
{\ar@{->}^{\id} "2";"4"};
{\ar@{=>}^{\alpha} (0,3)*{};(0,-3)};
\endxy .
\end{equation}
\end{definition}

\begin{definition}
2-cells in a double category $\D$ of the form \eqref{eq:globular} are
called {\sf globular}.
\end{definition}
\begin{remark}
Since a globular 2-cell is a 2-arrow in
the horizontal category $\cH(D)$ we may picture it  as $
\xy
(-10,0)*+ {c}="1"; 
(10,0)*+ {a}="2"; 
{\ar@/^1pc/^{\nu} "2";"1"};
{\ar@/_1pc/_{\mu} "2";"1"};
{\ar@{=>}^{\alpha} (0,2)*{};(0,-2)};
\endxy
$ .
\end{remark}
\noindent
In this paper we will use a number of double categories.  We start by
defining the double category of manifolds with corners, smooth maps
and set-theoretic relations (cf.\ \cite[Definition 3.9]{L1}).  The
definition of a category $\Man$ of manifolds with corners is reviewed
in Appendix~\ref{app:A}.  The reader should be aware that in the
differential-geometric literature there is a number of incompatible
definitions of smooth maps between manifolds with corners.  The
definition we use is probably the least restrictive.

 \begin{definition}[The double category $\RelMan^\Box$]
   \label{def:relmanbox}
   The objects of
   the double category $\RelMan^\Box$ are manifolds with corners.  The
   vertical arrows are the smooth maps.  Thus the category of objects 
   $(\RelMan^\Box)_0$ is the category $\Man$ of manifolds with corners
   and smooth maps.  A horizontal arrow $\mu:M\to N$ is a
   set-theoretic relation $\mu$ from $M$ to $N$,  that is a 
   subset of $M\times N$ (cf.\ Definition~\ref{notation:rel}).  A 2-cell
   $\alpha$ from a relation $\mu:M\to N$ to a relation $\nu: P\to Q$
   is a pair of smooth maps $(g: M\to P, f:N\to Q)$ so that
\[
(g\times f) (\mu) \subset \nu.
\]
 The composition in the arrows category $(\RelMan^\Box)_1$ is given by
 composing pairs of maps:
 \[
(\tau\xleftarrow{(k,l)} \nu) \circ (\nu\xleftarrow{(g,f)} \mu) = \tau
\xleftarrow{(k\circ g, l\circ f)} \mu.
  \] 
  The functor
  \[
\mathfrak{m}: (\RelMan^\Box)_1\times_{\mathfrak{s}, (\RelMan^\Box)_0, \mathfrak{t} } (\RelMan^\Box)_1\to (\RelMan^\Box)_1
\]
is defined by composing relations:
\[
\fm \left(   (\nu'\xLeftarrow{(f,h)}\mu', \nu\xLeftarrow{(g,f)} \mu
\right)
:= (\nu'\circ \nu\xLeftarrow{(g,h)} \mu'\circ \mu).
\]
\end{definition}
\begin{remark}\label{rmrk:RelMan}
The horizontal category $\cH(\RelMan^\Box)$ is the 2-category
$\RelMan$ of manifolds with corners, {\em arbitrary} (i.e., not
necessarily smooth) relations and inclusions of relations.
 \end{remark} 
The double category $\RelVect^\Box$ of vector spaces, linear maps and
linear relations is defined analogously to the definition of
$\RelMan^\Box$.  More formally we have:
\begin{definition}[The double category $\RelVect^\Box$]   The objects
  of the double category $\RelVect^\Box$ are (real) vector spaces.
  The vertical arrows are linear maps.  
Thus the category $(\RelVect^\Box)_0$ of objects  is the category $\Vect$ of
vector spaces and linear maps.
A horizontal arrow
  $\mu:V\to W$ is  a linear relation $\mu$ from a vector space $V$ to
  a vector space  $W$,
  which is  a vector subspace of $V\times W$.   A 2-cell $\alpha$ from
  a relation $\mu:V\to W$ to a relation $\nu: X\to Y$ is a pair of
  linear maps $(g: V\to X, f:W\to Y)$ so that
\[
(g\times f) (\mu) \subset \nu.
\]
 The composition in the category $(\RelVect^\Box)_1$ of arrows is given by
 composing pairs of maps:
 \[
(\tau\xleftarrow{(k,l)} \nu) \circ (\nu\xleftarrow{(g,f)} \mu) = \tau
\xleftarrow{(k\circ g, l\circ f)} \mu.
  \] 
  The functor
  \[
\mathfrak{m}: (\RelVect^\Box)_1\times_{\mathfrak{s}, (\RelMan^\Box)_0, \mathfrak{t} } (\RelVect^\Box)_1\to (\RelVect^\Box)_1
\]
is defined by composing relations:
\[ 
\fm \left(   (\nu'\xLeftarrow{(f,h)}\mu', \nu\xLeftarrow{(g,f)} \mu
\right)
:= (\nu'\circ \nu\xLeftarrow{(g,h)} \mu'\circ \mu).
\]
\end{definition}

\begin{remark}
It is easy to see that the horizontal category $\cH(\RelVect^\Box)$ is
the 2-category $\RelVect$ of vector spaces, linear relations and inclusions (Definition~\ref{def:relvect}).
\end{remark}

We will also need the double category $\SSub^\Box$ of surjective
submersions which was introduced in \cite{L2}.  To state the
definition of $\SSub^\Box$ it will be convenient to first recall the
notion of a map between two surjective submersions and also the notion
of an interconnection morphism.  We first introduce some 
notation which is motivated by control theory.

\begin{notation}
It will be convenient to denote a surjective submersion of
manifolds with corners by a single letter.  Thus a surjective
submersion $a$ consists of two manifolds with corners $a_\tot$ (the
``total space''), 
$a_\st$ (the ``state space'') and a surjective submersion
$p_a:a_\tot\to a_\st$.  We write $a= (a_\tot \xrightarrow{p_a}a_\st)$.
 \end{notation} 
  
\begin{definition}  \label{def:ssub} A {\sf morphism of surjective submersions} $f:a\to b$ is a
  pair of morphisms $f_\tot:a_\tot\to b_\tot$, $f_\st:a_\st\to b_\st$
  in the category $\Man$ of 
  manifolds with corners so that the diagram
  \[
\xy
(-10, 10)*+{b_\tot} ="1"; 
(10, 10)*+{a_\tot} ="2";
(-10,-5)*+{b_\st}="3";
(10, -5)*+{a_\st}="4";
{\ar@{->}_{f_\tot} "2";"1"};
{\ar@{->}^{p_b} "4";"3"};
{\ar@{->}_{p_a} "1";"3"};
{\ar@{->}^{f_\st} "2";"4"};
\endxy
\]
commutes $\Man$.   Note that the morphisms in the
category $\Man$ are $C^\infty$ maps, see Appendix~\ref{app:A}.

A morphism of surjective submersions $f:a\to b$ is an
{\sf interconnection morphism } if $f_\st:a_\st\to b_\st$ is a diffeomorphism.
\end{definition}
\begin{remark}
A motivation for the terminology of Definition~\ref{def:ssub} is
discussed in \cite{L2}. 
See  Example~\ref{ex:intercon} below.
 \end{remark} 

\begin{notation}\label{not:ssub}
Surjective submersions and their morphisms form a category which we
denote by $\SSub$.   It is a subcategory of the category of arrows of
$\Man$, the category of manifolds with corners and smooth maps.

Surjective submersions and interconnection morphisms form a
subcategory of $\SSub$.  We denote it by $\SSub^\inter$.
  \end{notation}

\begin{definition}[The double category $\SSub^\Box$ of
  surjective submersions]\label{def:ssub_box}
  The objects of the double category $\SSub^\Box$ are surjective
  submersions.  The horizontal 1-morphisms are maps/morphisms of surjective
  submersions (Definition~\ref{def:ssub} above). The vertical
  1-morphisms are interconnection morphisms (Definition~\ref{def:ssub}). 
    The 2-cells of $\SSub^\Box$ are commuting squares 
\[
  \xy
(-10,10)*+{c}="1";
(10, 10)*+{a}="2";
(-10,-5)*+{d}="3";
(10, -5)*+{b}="4";
{\ar@{<-}^{\mu} "1";"2"};
{\ar@{->}_{g} "1";"3"};
{\ar@{->}^{f} "2";"4"};
{\ar@{<-}_{\nu} "3";"4"};
  \endxy
\]
in the category   $\SSub$ of surjective submersions, 
where $\mu,\nu$ are maps of submersions and  $f$, $g$ are the interconnection
morphisms.  In particular,  the category of objects of the double
category $\SSub^\Box$ is the
category $\SSub^\inter$.
\end{definition}

We will need two versions of a ``functor between double categories:'' strict
and lax.    Functors between double categories  are often  referred to as ``double functors.''
\begin{definition}[Strict 1-morphism/map/double functor  of double categories]
  A strict {\sf morphism/map/double functor}  $F$ from a double category $\BB$ to a double
  category $\D$ is a pair of ordinary functors $F_0:\BB_0\to \D_0$,
  $F_1: \BB_1\to \D_1$ which commute strictly with the source, target
  and unit functors of the double categories $\BB$ and $\D$
(i.e., $\frs \circ F_1 = F_0 \circ \frs$, $\frt \circ F_1 = F_0 \circ
\frt$ and $\fru \circ F_0 = F_1 \circ \fru$)
  and
  strictly preserves the composition functor $\fm$: the diagram
  \begin{equation}\label{eq:2.18}
\xy
(-20, 10)*+{\BB_1\times_{\BB_0} \BB_1} ="1"; 
(20, 10)*+{\D_1\times_{\D_0}\D_1} ="2";
(-20,-6)*+{\BB_1 }="3";
(20, -6)*+{\D_1}="4";
{\ar@{->}^{F_1\times_{F_0}F_1} "1";"2"};
{\ar@{->}^{F_1} "3";"4"};
{\ar@{->}_{\fm} "1";"3"};
{\ar@{->}^{\fm} "2";"4"};
\endxy
    \end{equation}
strictly commutes in the 2-category $\CAT$ of categories.
\end{definition}
\begin{definition}[Normal lax 1-morphism/ lax functor of double
  categories] \label{def:lax-1-mor} A {\sf normal lax 1-morphism $F$} (or a
 {\sf  normal lax functor}) from a double category $\BB$ to a double
  category $\D$ consists of a pair of ordinary functors
  $F_0:\BB_0\to \D_0$, $F_1: \BB_1\to \D_1$ which commute strictly
  with the source, target and unit functors of the double categories
  $\BB$ and $\D$ together with a natural transformation
  $F_*: F_1\circ \fm \Rightarrow \fm \circ (F_1\times _{F_0} F_1)$
  which is subject to a  coherence condition  spelled out below.

The diagram  \eqref{eq:2.18} 2-commutes in $\CAT$ and the
  2-commutativity is witnessed by the natural transformation $F_*$.
  Thus for every pair $(\mu, \nu)$ of objects of $\BB_1$ composable
  under $\fm$ we have a (globular) 2-cell
  $(F_*)_{\mu,\nu}: F_1(\mu) *F_1 (\nu)\Rightarrow F_1(\mu*\nu)$. We
   further require  that for
  any triple of objects in $(\mu,\nu, \sigma)$ of $\BB_1$ composable
  under $\fm$
  \[
    (F_*)_{\mu*\nu, \sigma}\circ \left(
        (F_*)_{\mu,\nu} * \id_{F_1(\sigma)}        \right) =
      (F_*)_{\mu, \nu*\sigma}\circ \left(
\id_{F_1(\mu)}* (F_*)_{\nu, \sigma}
        \right)
      \]
      (recall that $\alpha *\beta := \fm(\alpha, \beta)$ for any two
      2-cells $\alpha, \beta$).
\end{definition}
\begin{remark}
  Since more general lax functors (the ones that are not necessarily normal) will
  make no appearance in this paper we will refer to normal lax
  functors simply as lax functors.  We trust that this will cause no confusion.
\end{remark}

Since double categories are categories internal to the (2-)category $\CAT$
of categories, there are two ways to internalize the notion of a
natural transformation between two double functors.
Namely given two double functors $F,G: \BB \to \DD$ between double categories
one can ask for a functor $\alpha:\BB_0\to \DD_1$ making an
appropriate diagram of categories and functors to commute (or
2-commute in the lax version).   This lead to the notion of a {\sf
  horizontal transformation}.  Alternatively since each double functor
is a pair of ordinary functors one can ask for a pair of (ordinary) natural
transformation.  This leads to the definition of a {\sf vertical
  transformation} which we now recall in the case where the double
functors are strict.  This  is the only case we need.
\begin{definition}[vertical transformation]
Let $F,G:\BB\to \DD$ be two strict double functors between two double
categories.   A {\sf vertical transformation} $\alpha$ from $F$ to
$G$ is a pair of natural transformations $\alpha_0: F_0 \Rightarrow
G_0$, $\alpha_1: F_1\Rightarrow  G_1$ (both often written as $\alpha$) subject
to the following conditions:
\begin{enumerate}
\item $\alpha$ is compatible with the source and
  target functors: for any object $\mu$ of $\BB_1$
  \[
    \frs ((\alpha_1)_\mu )= (\alpha_0)_{\frs(\mu)}, \qquad
    \frt ((\alpha_1)_\mu )= (\alpha_0)_{\frt(\mu)}  ;
    \]
\item $\alpha$ is compatible with the multiplication/composition functors
  $\fm$ of $\BB$ and $\DD$:
  for any pair $\mu, \nu$ of composable objects of $\BB_1$
  \[
(\alpha_1)_{\mu*\nu} = (\alpha_1)_\mu* (\alpha_1)_\nu;
\]
\item $\alpha$ is compatible with the unit functors:
  for any object $a$ of $\BB$
  \[
\fru ((\alpha_0)_a) = (\alpha_1)_{\fru(a)}.
   \] 
 \end{enumerate}
We write $\alpha:F\Rightarrow G$.
\end{definition}

\begin{remark} \label{rmrk:2.29}
Vertical transformations can be composed vertically component-wise:
given two vertical transformations $\alpha:F\Rightarrow G$, $\beta:
G\Rightarrow H$ we define $\beta\circ_v \alpha$ by setting 
\[
  (\beta\circ_v \alpha)_0 = \beta_0 \circ_v \alpha_0\qquad
  (\beta\circ_v \alpha)_1 = \beta_1 \circ_v \alpha_1.
  \]
 \end{remark} 

 We next explain a connection between control (open) systems and lax
 double functors.   The terms ``open system'' and ``control system''
 mean the same thing but have somewhat different connotations.
 ``Open'' is the opposite of ``closed,'' which means that energy and
 information can flow in and out of an open system.  ``Control''
 implies that we can affect the system's behavior.  
  Since in this paper we will not address the issues of
 shaping  behavior of systems, we prefer to call them open.  On the
 other hand, in several previous papers \cite{DL1}, \cite{L1} we
 used a functor from open systems to vector spaces and linear
 relations that we called $\Crl$.  Renaming it $\mathsf{Open}$ may
 well cause confusion especially since every topological space $Z$ has a
 category $\mathsf{Open} (Z)$ of open sets.   To avoid confusion and
 to preserve backward compatibility we will continue to call the
 functor in question  $\Crl$.
 
We now record  our  definition of a
control (open) system which was already mentioned in the introduction.

\begin{definition}\label{def:open_sys}
  An {\sf open} (or a {\sf
      control}) system  is a pair $(a, F)$ where  $a= (a_\tot
    \xrightarrow{p_a} a_\st)$ is a surjective submersion and $F:
    a_\tot \to Ta_\st$ is a smooth map
    such that $ F(q) \in T_{p_a (q)} a_\st$ for all $q\in a_\tot$.
      That is $\pi_{a_\st} \circ F = p_a$, where $\pi_{a_\st}:Ta_\st
      \to a_\st$ is the canonical submersion.
    \end{definition}
 \begin{remark} Given an open system $(a, F)$  we may also refer  to
   the  map $F:a_\tot \to Ta_\st$ as
   an {\sf  open system on the the surjective submersion $a$} or as an
   {\sf open system} (with the surjective submersion $a$ suppressed).
   \end{remark}   

    \begin{notation}[$\Crl(a)$] \label{not:crl}
      We denote the collection of all open system on a fixed submersion
      $a$ by $\Crl(a)$:
\[
\Crl(a):= \{F:a_\tot \to Ta_\st \mid \pi_{a_\st}\circ F = p_a\}
\]
  The collection $
\Crl(a) $ is a real vector space.
    \end{notation}

    \begin{definition} \label{def:Crl(f)} 
      Let $f:a\to b$ be a map of surjective
      submersions, that is,  a pair of
smooth maps $f_\tot :a_\tot \to b_\tot$, $f_\st:a_\st\to b_\st$ such
that the diagram
\[
\xy
(-10, 6)*+{a_\tot} ="1"; 
(10, 6)*+{b_\tot} ="2";
(-10,-10)*+{a_\st }="3";
(10, -10)*+{b_\st}="4";
{\ar@{->}_{p_a} "1";"3"};
{\ar@{->}^{p_b} "2";"4"};
{\ar@{->}^{f_\tot} "1";"2"};
{\ar@{->}_{f_\st} "3";"4"};
\endxy
\]
commutes.  We define a linear relation $\Crl(f):\Crl(a) \to  \Crl(b)$ by 
\[
  \Crl(f) := \left\{ (F,G)\in \Crl(a)\times \Crl(b) \left|\quad  
\xy
(-10, 6)*+{a_\tot} ="1"; 
(10, 6)*+{b_\tot} ="2";
(-10,-10)*+{Ta_\st }="3";
(10, -10)*+{Tb_\st}="4";
{\ar@{->}_{F} "1";"3"};
{\ar@{->}^{G} "2";"4"};
{\ar@{->}^{f_\tot} "1";"2"};
{\ar@{->}_{Tf_\st} "3";"4"};
\endxy \quad \textrm{commutes} \right \}. \right. 
\]
\end{definition}
The next definition says more or less the same thing as
Definition~\ref{def:Crl(f)} but from a somewhat different point of
view.
\begin{definition}
 Let $f:a\to b$ be a map of surjective
      submersions.  Suppose $w\in \Crl(b)$ and $u\in \Crl(a)$ are two
      open systems so that $(u,w)\in \Crl(f)$, i.e., the diagram
      \[
 \xy
(-10, 6)*+{a_\tot} ="1"; 
(10, 6)*+{b_\tot} ="2";
(-10,-10)*+{Ta_\st }="3";
(10, -10)*+{Tb_\st}="4";
{\ar@{->}_{u} "1";"3"};
{\ar@{->}^{w} "2";"4"};
{\ar@{->}^{f_\tot} "1";"2"};
{\ar@{->}_{Tf_\st} "3";"4"};
\endxy       
\]
commutes.  We then say that the control system $u$ {\sf is
  $f$-related} to the control system $w$.
  \end{definition}
\begin{definition} \label{def:2.35}
Suppose $\varphi:a\to b$ is an interconnection morphism  (Definition~\ref{def:ssub}).  Then $\varphi$ induces a
linear map $\varphi^*: \Crl(b) \to \Crl (a)$.  It is defined by
\[
\varphi^*F:= T(\varphi_\st)\inv \circ F \circ \varphi_\tot
\]
for all open systems $F\in \Crl(b)$.
\end{definition}
\begin{remark}
For an interconnection morphism  $\varphi:a\to b$ between two
surjective submersions  the linear relation $\Crl(\varphi): \Crl(a) \to \Crl(b)$ is given
by 
\[
\Crl(\varphi):  =\{(\varphi^*F, F) \in \Crl(a)\times \Crl(b)\mid 
F\in \Crl(b)\} 
\]
where the linear map $\varphi^*$ is defined above (Definition~\ref{def:2.35}).
Consequently the linear relation $\Crl(\varphi)$ is the transpose of the graph of the
 map $\varphi^*$ (see Definition~\ref{def:transpose}).
\end{remark}
\begin{remark} \label{rmrk:st_conv}
As is traditional in differential geometry and geometric mechanics all
manifolds with corners are nonempty unless specifically mentioned otherwise.
  \end{remark}
\begin{example} \label{ex:intercon}
  Let $U$ and $M$ be two manifolds with corners.  The projection on
  the second factor $p:U\times M\to M$ is a surjective submersion
  (note 
  Remark~\ref{rmrk:st_conv}).  So
  is the identity map $\id_M:M\to M$.  Consider a map of surjective
  submersions
  $\varphi= (\varphi_\tot, \varphi_\st): (M\xrightarrow{\id_M}M) \to
  (U\times M\xrightarrow{p} M) $.  The map
  $\varphi_\tot:M\to U\times M$ then has to be of the form
  $\varphi_\tot(m) = (h(m), \varphi_\st(m))$ where $h:M\to U$ is a
  smooth map.  In particular if $\varphi_\st = \id_M$ then for any map
  $h$, $\varphi = ((h, \id_M), \id_M)$ is an interconnection morphism.
  The induced map
  $\varphi^*: \Crl(U\times M\to M) \to \Crl(M\xrightarrow{\id} M) =
  \scrX(M)$ is given by
  \[
(\varphi^*F ) (m) = F(h(m), m).
\]
(Recall that $\scrX(M)$ denotes the space of vector fields on the
manifold $M$.)  To summarize: given an open system $F:U\times M\to TM$
we produce from $F$ a vector field $X$ on the manifold $M$ by plugging
the values of the states of the open system into the controls $U$ by
means of the map $h$.
  \end{example}
We observe that pullbacks by interconnection maps preserve relations
between open systems.
\begin{lemma}
Let $f:a\to c$, $g:b\to d$ be maps of surjective submersions,
$\varphi: a\to b$, $\psi: c\to d$ two interconnection maps so that the
diagram
\[
\xy
(-10,6)*+{c}="1";
(10, 6)*+{a}="2";
(-10,-7)*+{d}="3";
(10, -7)*+{b}="4";
{\ar@{<-}^{f} "1";"2"};
{\ar@{->}_{\psi} "1";"3"};
{\ar@{->}^{\varphi} "2";"4"};
{\ar@{->}^{g} "4";"3"};
\endxy
\]
commutes in the category $\SSub$ of surjective submersions (i.e., the diagram is a
2-cell in the double category $\SSub^\Box$).  If two open systems
$F\in \Crl(b)$ and $G\in \Crl(d)$ are $g$-related then the
interconnected systems $\varphi^*F$ and $\psi^*G$ are $f$-related.
Consequently
  \[
 \xy
(-10,6)*+{\Crl(c)}="1";
(10, 6)*+{\Crl(a)}="2";
(-10,-7)*+{\Crl(d)}="3";
(10, -7)*+{\Crl(b)}="4";
{\ar@{<-}^{\Crl(f)} "1";"2"};
{\ar@{<-}_{\psi^*} "1";"3"};
{\ar@{<-}^{\varphi^*} "2";"4"};
{\ar@{->}^{\Crl(g)} "4";"3"};
{\ar@{<=}^{} (0,2)*{};(0,-2)};
    \endxy
      \]
      is a 2-cell in the double category $\RelVect^\Box$ of vector
      spaces, linear maps and linear relations.
 \end{lemma} 
 \begin{proof}
See \cite[Lemma~8.12]{L2}
\end{proof}

\begin{lemma}
The mapping from the 2-cells of the double category $\SSub^\Box$ of
surjective submersions to the double category $\RelVect^\Box$ given by
\[
    \Crl \left(
\xy
(-10,6)*+{c}="1";
(10, 6)*+{a}="2";
(-10,-7)*+{d}="3";
(10, -7)*+{b}="4";
{\ar@{<-}^{f} "1";"2"};
{\ar@{->}_{\psi} "1";"3"};
{\ar@{->}^{\varphi} "2";"4"};
{\ar@{->}^{g} "4";"3"};
{\ar@{=>}^{} (0,2)*{};(0,-2)};
\endxy
\right) :=
\xy
(-10,6)*+{\Crl(c)}="1";
(10, 6)*+{\Crl(a)}="2";
(-10,-7)*+{\Crl(d)}="3";
(10, -7)*+{\Crl(b)}="4";
{\ar@{<-}^{\Crl(f)} "1";"2"};
{\ar@{<-}_{\psi^*} "1";"3"};
{\ar@{<-}^{\varphi^*} "2";"4"};
{\ar@{->}^{\Crl(g)} "4";"3"};
{\ar@{<=}^{} (0,2)*{};(0,-2)};
    \endxy
   \] 
  defines a lax (contravariant) 1-morphism of double categories
\[
  \Crl: (\SSub ^\Box)^\op \to \RelVect^\Box.
\] 
 \end{lemma}

 \begin{proof} Note  that $\Crl(id_c) = id_{\Crl(c)}$. Next  observe
   that given a pair of compatible maps of submersions $c\xrightarrow{f} b
\xrightarrow{g} a$  the composite
$\Crl(f)\circ \Crl(g)$ of linear relations is a subspace of the linear
relation $\Crl(f\circ g)$.  The inclusion  $
\Crl(f)\circ \Crl(g)\hookrightarrow \Crl(f\circ g)$ is the $f,g$
component $(\Crl_*)_{f,g}$
of the desired natural transformation $(\Crl)_* $ (cf.\ Definition~\ref{def:lax-1-mor}).
  \end{proof}

  \subsection{Categories of lists}\mbox{}

 In this somewhat technical subsection we collect a number of
 definitions and facts that will provide a convenient language later
 on.  Recall that any set $X$ can be considered as a discrete category.
 Then a functor $\tau:X\to \sfC$ from a set $X$ (considered as a
 category) to a category
 $\sfC$ assigns to each element $x\in X$ an object $\tau(x)$ of
 $\sfC$.  Thus   a functor from a set  to a category $\sfC$ is an
 unordered list of elements of $\sfC$ (possibly with repetitions)
 which is indexed by the elements of the set.
 Functors from various sets to a fixed category $\sfC$ can be
 assembled into a category.

\begin{definition}[The category of lists $\Set/\sfC$ in a category
  $\sfC$]\label{def:cat_of_lists}
  Fix a category $\sfC$.  An object of the category of lists
  $\Set/\sfC$ is a functor $\tau:X\to \sfC$ where $X$ is a set thought of
  as a discrete category.   A morphism in $\Set/\sfC$ from a functor
  $\tau:X\to \sfC$ to a functor $\tau':X'\to \sfC$ is a map of sets
  $\varphi:X\to X'$  so that the triangle of functors 
\[
\xy
(-10, 10)*+{X} ="1"; 
(10, 10)*+{X'} ="2";
(0,-2)*+{\sfC }="3";
{\ar@{->}_{\tau} "1";"3"};
{\ar@{->}^{\varphi} "1";"2"};
{\ar@{->}^{\tau'} "2";"3"};
\endxy 
\]
commutes.
The composition is defined by pasting of the triangles:
\[
  \left(
\xy
(-10, 6)*+{X''} ="1"; 
(10, 6)*+{X'} ="2";
(0,-6)*+{\sfC }="3";
{\ar@{->}_{\tau''} "1";"3"};
{\ar@{<-}^{\psi} "1";"2"};
{\ar@{->}^{\tau'} "2";"3"};
\endxy 
  \right)\circ \left(
\xy
(-10, 6)*+{X'} ="1"; 
(10, 6)*+{X} ="2";
(0,-4)*+{\sfC }="3";
{\ar@{->}_{\tau'} "1";"3"};
{\ar@{<-}^{\varphi} "1";"2"};
{\ar@{->}^{\tau} "2";"3"};
\endxy 
\right) =
\xy
(-10, 6)*+{X''} ="1"; 
(10, 6)*+{X} ="2";
(0,-6)*+{\sfC }="3";
{\ar@{->}_{\tau''} "1";"3"};
{\ar@{<-}^{\psi\circ \varphi} "1";"2"};
{\ar@{->}^{\tau} "2";"3"};
\endxy 
  \]
\end{definition}

\begin{definition}[The category $\FinSet/\sfC$ of finite lists]
The category of lists $\Set/\sfC$ has a subcategory $\FinSet/\sfC$
whose objects are {\em finite} lists, i.e., functors from finite sets.
\end{definition}

\begin{remark}
If a category $\sfC$ has coproducts then there is a canonical functor
$\rotpi: \Set/\sfC \to \sfC$ defined on objects by taking colimits
\[
  \rotpi (X\xrightarrow{\tau}\sfC) := \colim (X\xrightarrow{\tau}
  \sfC ) = \bigsqcup _{x\in X}\tau (x).
\]
On arrows it is defined by the universal properties of coproducts:
given a morphism $\varphi: (X\xrightarrow{\tau} \sfC)\to
(X'\xrightarrow{\tau'} \sfC)$ the morphism $\rotpi (\varphi): \rotpi
(\tau)\to \rotpi (\tau')$ is the unique arrow making the diagram
\[
  \xy
(-10, 6)*+{\rotpi(\tau)} ="1"; 
(12, 6)*+{\rotpi(\tau')} ="2";
(-10,-6)*+{\tau(x)}="3";
(12, -6)*+{\tau'(\varphi(x)))}="4";
{\ar@{<--}_{\rotpi(\varphi)} "2";"1"};
{\ar@{->}_{\id} "3";"4"};
{\ar@{->}^{\imath_x} "3";"1"};
{\ar@{->}_{\imath_{\varphi(x)}} "4";"2"};
  \endxy
 \] 
commute.  Here $\imath_x$ and $\imath_{\varphi(x)}$  are  the
canonical inclusions.
\end{remark}

\begin{remark}\label{rmrk:2.35}
If the category $\sfC$ has finite products then there is  a canonical functor
$\Pi: (\FinSet/\sfC)^\op \to \sfC$ defined on objects by taking
limits:
\[
\Pi (\tau): = \lim (X\xrightarrow{\tau} \sfC) = \bigsqcap _{x\in X}\tau(x).
\]
On arrows it is defined by the universal property of products.  Namely, given
a morphism $\varphi: (X\xrightarrow{\tau} \sfC)\to
(X'\xrightarrow{\tau'} \sfC)$ the morphism $\Pi (\varphi): \Pi
(\tau')\to \Pi (\tau)$ is the unique arrow making the diagram
\[
  \xy
(-16, 6)*+{\Pi(\tau)} ="1"; 
(12, 6)*+{\Pi(\tau')} ="2";
(-16,-6)*+{\tau(x)}="3";
(12, -6)*+{\tau'(\varphi(x)))}="4";
{\ar@{-->}_{\Pi(\varphi)} "2";"1"};
{\ar@{<-}_{\id} "3";"4"};
{\ar@{<-}^{\pi_x} "3";"1"};
{\ar@{<-}_{\pi_{\varphi(x)}} "4";"2"};
  \endxy
 \] 
commute.  Here $\pi_x$ and $\pi_{\varphi(x)}$  are  the
canonical projections.
\end{remark}

Recall that the category $\Man$ of manifolds with corners has finite
products.  This can be seen as follows.  The terminal object is a one
point manifold.  The Cartesian product of two manifolds with corners
is again a manifold with corners and serves (together with the
canonical projections maps) as their product in the category $\Man$.
Since $\Man$ has a terminal object and binary products, it has finite products.

\begin{example} \label{ex:exPiphi1}
  Let $\sfC =\Man$, the category of manifolds with corners, $X = Y = \{1,2, 3\}$,
  $\varphi:X\to Y$ be given by
  \[
\varphi(1) = 2,\qquad \varphi(2) = 1, \qquad \varphi(3) =2.
  \] 
  Fix a manifold with corners $A$. Let $\tau, \mu: X, Y\to \Man$ be
  the constant maps defined by $\tau(j) = \mu(j) =A$ for all $j$.
  Then $\varphi: \tau \to \mu$ is a morphism in $\FinSet/\Man$ and
  $\Pi (\varphi):\Pi(\mu) = A^3 \to A^3 = \Pi(\tau)$ is given by
\[
\Pi(\varphi)\, (a_1, a_2, a_3) = (a_2, a_1, a_2)
\]
for all $(a_1,a_2, a_3) \in A^3 = \Pi(\mu)$.
\end{example}

The categories of lists $\Set/\sfC$ and $\FinSet/\sfC$ have variants
in which the commuting triangles of morphisms are replaced by
2-commuting triangles.  More precisely we have the following
definitions.

\begin{definition}[The category of lists $(\Set/\sfC)^\Rightarrow$]
\label{def:2.40}
Fix a category $\sfC$.  An object of the category of lists
$(\Set/\sfC)^\Rightarrow $ is a functor $\tau:X\to \sfC$ where $X$ is
a set thought of as a discrete category.  That is, the objects of
$(\Set/\sfC)^\Rightarrow$ are the same as the objects of $\Set/\sfC$.
A morphism in $(\Set/\sfC)^\Rightarrow$ from a functor
$\tau:X\to \sfC$ to a functor $\tau':X'\to \sfC$ is a 2-commuting
triangle
  \[
\xy
(-10, 10)*+{X} ="1"; 
(10, 10)*+{X'} ="2";
(0,-2)*+{\sfC }="3";
{\ar@{->}_{\tau} "1";"3"};
{\ar@{->}^{\varphi} "1";"2"};
{\ar@{->}^{\tau'} "2";"3"};
{\ar@{<=}_{\scriptstyle \Phi} (4,6)*{};(-0.4,4)*{}} ; 
\endxy .
\]
In other words a morphism in $(\Set/\sfC)^\Rightarrow$ is pair
$(\varphi, \Phi)$ where $\varphi:X\to X'$ is a map of sets and
$\Phi: \tau \Rightarrow \tau' \circ \varphi$ is a natural
transformation.  Given a pair of composable morphisms
$(\varphi, \Phi): (X\xrightarrow{\tau}\sfC)\to (X'\xrightarrow{\tau'}
\sfC)$ and
$(\psi, \Psi) : (X'\xrightarrow{\tau'} \sfC)\to
(X''\xrightarrow{\tau''} \sfC)$ their composition is defined by pasting
of the 2-commuting triangles.  That is,
  \[
    (\psi, \Psi)  \circ (\varphi, \Phi) :=
(\psi  \circ \varphi, (\Psi \circ \varphi) \circ _v \Phi),
    \]
where 
 $\Psi \circ \varphi: \tau \Rightarrow \tau''$ is the whiskering
of the natural transformation with a functor and $\circ_v$ denotes the
vertical composition of natural transformations.
\end{definition}

\begin{remark}
  Observe that the category $\Set/\sfC$ of lists is a subcategory of
  the category $(\Set/\sfC)^\Rightarrow$
  \end{remark}
\begin{remark} \label{rmrk:2.43}
  If the category $\sfC$ has coproducts then there is a canonical
  functor $\rotpi: (\Set/\sfC)^\Rightarrow \to \sfC$ which extends the
  functor $\rotpi: \Set/\sfC \to \sfC$.  As before to each object
  $\tau:X\to \sfC$ the functor $\rotpi$ assigns the coproduct
  $\bigsqcup _{x\in X}\tau (x)$.  On arrows $\rotpi$ is again defined
  by the universal properties of coproducts: given a morphism
  $(\varphi, \Phi): (X\xrightarrow{\tau} \sfC)\to
  (X'\xrightarrow{\tau'} \sfC)$ the morphism
  $\rotpi (\varphi, \Phi): \rotpi (\tau)\to \rotpi (\tau')$ is the
  unique arrow making the diagram
\[
  \xy
(-10, 6)*+{\rotpi(\tau)} ="1"; 
(12, 6)*+{\rotpi(\tau')} ="2";
(-10,-6)*+{\tau(x)}="3";
(12, -6)*+{\tau'(\varphi(x)))}="4";
{\ar@{<--}_{\rotpi(\varphi, \Phi)} "2";"1"};
{\ar@{->}_{\Phi_x} "3";"4"};
{\ar@{->}^{\imath_x} "3";"1"};
{\ar@{->}_{\imath_{\varphi(x)}} "4";"2"};
  \endxy
 \] 
commute.  Here as before $\imath_x$ and $\imath_{\varphi(x)}$  denote  the
canonical inclusions.
  \end{remark}

  \begin{definition}[The category $(\FinSet/\sfC)^\Leftarrow$]  The
    objects of the category $(\FinSet/\sfC)^\Leftarrow$ are finite
    lists $\tau:X\to \sfC$ (i.e., they are the same as the objects of
    the category $\FinSet/\sfC$ of finite lists in the category
    $\sfC$).  A morphism $(\varphi,\Phi): (X\xrightarrow{\tau} \sfC )
    \to (Y\xrightarrow{\mu}\sfC)$ is a pair $(\varphi, \Phi)$ where
    $\varphi: X\to Y$ is a map of sets and $\Phi$ is now a natural
    transformation from $\mu \circ \varphi$ to $\tau$.  That is, the
    morphism is a 2-commuting triangle of the form
    \begin{equation}
\xy
(-10, 10)*+{X} ="1"; 
(10, 10)*+{Y} ="2";
(0,-2)*+{\sfC }="3";
{\ar@{->}_{\tau} "1";"3"};
{\ar@{->}^{\varphi} "1";"2"};
{\ar@{->}^{\mu} "2";"3"};
{\ar@{=>}_{\scriptstyle \Phi} (4,6)*{};(-0.4,4)*{}} ; 
\endxy .
\end{equation}
The composition in $(\FinSet/\sfC)^\Leftarrow$ is defined by pasting
of the triangles.
\end{definition}
\begin{remark} \label{rmrk:2.41}
If the category $\sfC$ has finite products then the functor $\Pi:
(\FinSet/\sfC)^\op \to \sfC$ extends to a functor $\left(
  (\FinSet/\sfC)^\Leftarrow\right)^\op \to \sfC$, which we again denote by
  $\Pi$: given a morphism $(\varphi, \Phi): (X\xrightarrow{\tau} \sfC)\to
(X'\xrightarrow{\tau'} \sfC)$ the morphism $\Pi (\varphi, \Phi): \Pi
(\tau')\to \Pi (\tau)$ is the unique arrow making the diagram
\[
  \xy
(-10, 6)*+{\Pi(\tau)} ="1"; 
(12, 6)*+{\Pi(\tau')} ="2";
(-10,-6)*+{\tau(x)}="3";
(12, -6)*+{\tau'(\varphi(x)))}="4";
{\ar@{-->}_{\Pi(\varphi, \Phi )} "2";"1"};
{\ar@{<-}_{\Phi_x} "3";"4"};
{\ar@{<-}^{\pi_x} "3";"1"};
{\ar@{<-}_{\pi_{\varphi(x)}} "4";"2"};
  \endxy
 \] 
commute.  Here as before $\pi_x$ and $\pi_{\varphi(x)}$  are  the
canonical projections.
\end{remark}

\begin{example} \label{ex:exPiphi2}
  Let $\sfC =\Man$, the category of manifolds with corners, $X = Y = \{1,2, 3\}$.
  Fix two  manifolds with corners $A$ and $B$ and   a smooth map $s:A\to B$.  Let $\tau, \mu: X, Y\to \Man$ be
  the constant maps defined by $\tau(j) = B$, $\mu(j) =A$ for all
  $j$.    We define a morphism $(\varphi,
  \Phi): \tau \to \mu$ in $(\FinSet/\Man)^\Leftarrow$ as follows.
As before we define $\varphi$ by 
  \[
\varphi(1) = 2,\qquad \varphi(2) = 1, \qquad \varphi(3) =2.
\]
We define $\Phi_i: \mu (\varphi(i)) = A \to \tau(i) = B$ to be the map
$s:A \to B$ for all $i$.
  Then $(\varphi,\Phi): \tau \to \mu$ is a morphism in $(\FinSet/\Man)^\Leftarrow$ and
  $\Pi (\varphi, \Phi):\Pi(\mu) = A^3 \to B^3 = \Pi(\tau)$ is given by
\[
\Pi(\varphi,\Phi)\, (a_1, a_2, a_3) = (s(a_2), s(a_1), s( a_2))
\]
for all $(a_1,a_2, a_3) \in A^3 = \Pi(\mu)$.
\end{example}

\mbox{}\\

\section{The category of elements $\int  F$ for  a functor
  with values in linear relations}\label{sec:cat_of_el}

Given a set-valued functor $F:\sfC \to \Set$ on a category $\sfC$
there is a well-known construction due to Grothendieck that produces
a {\sf category of elements}  $\int F $ together with the functor $\pi_F: \int
F\to \sfC$.  Recall that the objects of the
category $\int F$ are pairs $(c, x)$ where $c$ is an object of $\sfC$
and $x$ is an element of the set $F(c)$.  A morphism in $\int F$ from
$(c, x)$ to $(c', x')$ is a morphism $h:c\to c'$ in $\sfC$ such that
$F(h)x = x'$.  The functor $\pi_F:\int F\to \sfC$  maps an
arrow $h:(x,c)\to (x',c')$ of $\int F$ to the arrow $h:c\to c'$ of $\sfC$.  The
functor $\pi_F: \int F \to \sfC$ has nice lifting properties.

It will be useful for us to have a generalization of this construction
to functors with values in the 2-category $\RelVect$ of vector spaces
and linear relations. 
Namely suppose $\sfC$ is a category and we are given a (lax) functor
$F: \sfC \to \RelVect$.  That is, suppose that for any pair of
composable arrows $c''\xleftarrow{g} c'\xleftarrow{h}c$ in $\sfC$ we
have an inclusion $F(g) \circ F(h) \subset F(g\circ h)$.  We then can
define the ``category of elements'' $\int F$ and a functor
$\pi_F:\int F\to \sfC$ (see below).  In general the functor $\pi_F$ has no evident lifting
properties.  In the next section we will realize the category $\HyDS$
of hybrid dynamical systems as the category of elements for a lax
$\RelVect$-valued functor.  Later we will use the construction to
produce the category $\HyOS$ of hybrid open systems.

\begin{definition}[The category of elements $\int F$ of a lax  functor $F: \sfC
  \to \RelVect$] \label{def:cat_of_elements}
  Let $\sfC$ be a category and $F:\sfC \to \RelVect$ a lax 2-functor
  with values in the 2-category $\RelVect$ of linear relations.
  We define the {\sf category of elements $\int F$} of the functor $F$ as follows.
\begin{enumerate}
\item The objects of $\int F$ are pairs $(c, x)$ where $c$ is an
  object of $\sfC$ and $x$ is a vector in the vector space $F(c)$.
 \item A morphism from $(c,x)$ to $(c',x')$ is a morphism $h: c\to c'$
   such that $(x,x')$ is an element of the linear relation $F(h) \subset
   F(c)\times F(c')$. 
 \end{enumerate}
 It is easy to see that $\int  F$ is a category.  We also have  a
 functor $\pi_F: \int F\to \sfC$ which is defined by
\[
\pi_F ((c,x)\xrightarrow{h} (c',x') ) = c\xrightarrow{h}c'.
\]
\end{definition}

\begin{remark}
Recall that continuous time dynamical systems form a category $\DS$.
The objects of this category are pairs $(M, X)$ where $M$ is a
manifold with corners and $X\in \scrX(M)$ is a vector field on $M$.  A
morphism from a dynamical system $(M,X)$ to a system $(N,Y)$ is a so
called ``semi-conjugacy'': a smooth map $f:M\to N$ so that $Tf\circ X
= Y\circ f$.
  \end{remark}

 \begin{example}[Continuous time dynamical systems from the vector
   field functor $\scrX:\Man\to \RelVect$]\label{ex:3.3}
Consider the category $\Man$ of manifolds with corners.   The
map $\scrX$ that assigns to every manifold $M$ the vector space
$\scrX(M)$ of vector fields on $M$ extends to a lax functor $\scrX: \Man
\to \RelVect$: given a smooth map $f:M\to N$ the relation $\scrX(f)$
is, by definition
\[
\scrX(f) := \{ (X,Y)\in \scrX(M)\times \scrX(N) \mid Y \circ f = Tf \circ X\}.
\]
The category of elements $\int \scrX$ is 
the category $\DS$ of continuous
time dynamical systems.
\end{example}

\begin{remark}
The category $\OS$ of (continuous time) open systems is less well
known than the category $\DS$ but it has appeared in literature. See
for example \cite{TP} where the category of open systems is called
$\mathsf{Con}$.
The objects of the category $\OS$ are pairs $(a, f)$ where $a=(p_a: a_\tot
\to a_\st)$ is a surjective submersion and $f:a_\tot \to Ta_\st$ is an
open system.
 \end{remark}  

\begin{example} [The category $\OS$ of open systems from the functor
  $\Crl:\SSub \to \RelVect$] \label{ex:3.4}
Consider the category $\SSub$ of surjective submersions
(Notation~\ref{not:ssub}).  Recall that objects of $\SSub$ 
are surjective submersions $a = (p_a: a_\tot \to a_\st)$ of manifolds
with corners and that a morphism
from a submersion $a$ to a submersion $b$ is, by definition, a pair of
smooth maps $f_\tot :a_\tot \to b_\tot$, $f_\st:a_\st\to b_\st$ such
that the diagram
\[
\xy
(-10, 10)*+{a_\tot} ="1"; 
(10, 10)*+{b_\tot} ="2";
(-10,-10)*+{a_\st }="3";
(10, -10)*+{b_\st}="4";
{\ar@{->}_{p_a} "1";"3"};
{\ar@{->}^{p_b} "2";"4"};
{\ar@{->}^{f_\tot} "1";"2"};
{\ar@{->}_{f_\st} "3";"4"};
\endxy
\]
commutes.

For every surjective submersion $a$ we have the vector space $\Crl(a)$
of  open systems on $a$ (Definition~\ref{def:open_sys}).
For each map of
submersions $f:a\to b$,  there is a linear relation
$\Crl(f) \subset \Crl(a)\times \Crl(b)$ consisting of $f$-related
control systems (Definition~\ref{def:Crl(f)}).
The corresponding category of elements $\int  \Crl$ is   the category
$\OS$ of
open (control) systems.
\end{example}
\begin{remark}
Note that there is a canonical embedding $\imath: \Man \to \SSub$. On
objects it is given by $\imath(M) = (M\xrightarrow{\id_M} M)$.  Note that the
functors $\scrX$ and $\Crl$ are compatible with the embedding
$\imath$:
\begin{equation}
\Crl \circ \imath = \scrX.
\end{equation}
In particular we can consider every vector field $X:M\to TM$ on a
manifold $M$ as a control system on the submersion $\id_M:M\to M$.  We
think of vector fields as control systems with no inputs, that is, as
closed systems.  Thus, somewhat paradoxically, every closed system is
an open system (with no inputs from ``the outside'').
\end{remark}

We end the section with another example.  Its usefulness at the moment
may not be apparent and a reader is welcome to skip it for the time
being. Later on we will generalize
this example to hybrid systems.  See Example~\ref{ex:hds_on_product} below.

\begin{example}\label{rmrk:vf_on_product} In this example we show that
a vector field $X$ on a product of two manifold $M_1\times M_2$ is the
result of interconnection of two open systems.  That is, we claim that
given a vector field $X\in \scrX(M_1\times M_2)$ there exist two open
systems $(M_1\times M_2\to M_1, X_1: M_1\times M_2\to TM_1)$,
$(M_1\times M_2\to M_2, X_2: M_1\times M_2\to TM_2)$ and an 
interconnection morphism (Definition~\ref{def:ssub})
\[
\varphi:  (M_1\times M_2
\xrightarrow{\id}  M_1\times M_2) \to (M_1\times M_2\to M_1)\times (M_1\times M_2\to M_2)
\]
so that $\varphi^*(X_1\times X_2) = X$ where
\[
\varphi^*:
\Crl\left((M_1\times M_2\to M_1)\times (M_1\times M_2\to M_2) \right)\to \Crl(M_1\times M_2
\xrightarrow{\id}  M_1\times M_2) = \scrX(M_1\times M_2)
\]
is the linear
map induced by $\varphi$ (see Definition~\ref{def:2.35}).

This can be seen as follows.
   The vector field $X:M_1\times M_2 \to TM_1\times TM_2$ is of the form $X = (X_1,
X_2)$. The maps $X_1:M_1\times M_2\to TM_1$, $X_2:M_1\times M_2\to TM_2$
are open systems on the submersions $M_1\times M_2\to M_1$, $M_1\times
M_2\to M_2$, respectively.   The map 
\[
  \varphi: (M_1\times M_2 \to M_1\times M_2) \to \left(M_1\times M_2\to M_1
  \right) \times \left(M_1\times M_2\to M_2)
    \right)
   \]
with $\varphi_\st = \id_{M_1\times M_2}$ and $\varphi_\tot:M_1\times
M_2 \to (M_1\times M_2)\times (M_1\times M_2) $ given by
\[
\varphi_\tot (m_1, m_2) = ((m_1, m_2), (m_1, m_2)).
\]
is an interconnection map.  Furthermore
  \[
X(m_1, m_2) = (X_1 (m_1, m_2), X_2 (m_1, m_2)) = ((X_1\times X_2) \circ
\varphi_\tot )\, (m_1, m_2)
\]
for all $(m_1, m_2) \in M_1\times M_2$. 

We conclude that $X= (X_1,X_2) $ is a vector field on the product
$M_1\times M_2$ if and only if $X = \varphi^* (X_1\times X_2)$ where
$X_1, X_2$ are open systems, $\varphi$ is the interconnection map
defined above and $\varphi^*:\Crl \left(M_1\times M_2\to M_1
  \right) \times \left(M_1\times M_2\to M_2) \right)\to \Crl (M_1\times M_2
    \to M_1\times M_2) = \scrX(M_1\times M_2)$ is the linear map
    induced by $\varphi$.
\end{example} 
\mbox{}

\section{A category of hybrid phase spaces $\HyPh$} 

We now construct the category $\HyPh$ of hybrid phase spaces.  We have
a number of reasons for introducing this notion.  First of all,
traditional definitions of hybrid dynamical systems involve a lot of
data.  We would like to organize these data in a compact and
structured way.  Secondly, in the next section we will need to define
hybrid open systems.  There seems to be no consensus in the literature
of what a hybrid open system should be.  Our approach is to view
hybrid open systems as analogous to continuous time open systems.  As
we mentioned in the previous section, it is convenient to view a
continuous time open system as a pair $(a, F)$ where
$a = a_\tot\to a_\st$ is a surjective submersion and
$F:a_\tot \to Ta_\st$ is an open system on the submersion.  We will
define a hybrid open system to be a pair $(a, F)$ where now
$a = a_\tot \xrightarrow{p_a} a_\st$ is a hybrid surjective
submersion, that is, a certain map of hybrid phase spaces, and
$F:\UU(a_{tot})\rightarrow T\UU(a_{st})$ is a continuous time open
system on an associated surjective submersion $\UU (a)$ (see
Definition~\ref{def:HyOS}).  Additionally we want our definition of a
hybrid phase space to meet a number of requirements and pass a few
sanity tests.  Before we spell them out, we would like  to have
two examples of traditionally defined hybrid dynamical systems before
us. For the reader's convenience a traditional definition of a hybrid
dynamical system is reviewed in Appendix~\ref{app:B}.

\begin{example}[A room with a heater and a thermostat]
\label{ex:1.1}
Imagine a one room house in winter.  The room has a heater and a
thermostat.  For convenience we choose the units of temperature so
that the comfort range in the room falls between 0 and 1 (say
$0 = 18^\circ C$ and $1= 20^\circ C$).  Suppose the room starts at the
temperature $x=1$ and cools down to $x=0$.  Assume that the evolution
of temperature is governed by the equation $\dot{x} = -1$.  Once the
temperature drops down to 0 the thermostat turns on the heater and the
temperature evolution is now governed by $\dot{x} = 1$.  The dynamical
system we have just described is one of the simplest examples of a
hybrid dynamical system.  Formally the system consists of the disjoint
union of two manifolds with boundary
$M_{\textrm{on}}=M_{\textrm{off}} =[0,1]$ with a vector field $X $ on
$M = M_{\textrm{on}}\sqcup M_{\textrm{off}}$ defined by
$X|_{M_{\textrm{off}}} = -\frac{d}{dx}$ and
$X|_{M_{\textrm{on}}} = \frac{d}{dx}$.  Additionally we have two partial
functions (see Notation~\ref{notation:rel}):
$f:M_{\textrm{off}} \to M_{\textrm{on}}$ which takes
$0\in M_{\textrm{off}}$ to $0\in M_{\textrm{on}}$ (and is undefined
elsewhere) and $g:M_{\textrm{on}} \to M_{\textrm{off}}$ which takes
$1\in M_{\textrm{on}} $ to $1 \in M_{\textrm{off}}$ (and is undefined
elsewhere).  The labelled directed graph $ \xy
(-10,0)*{\bullet}="1"; 
(10,0)*{\bullet}="2"; 
(-15,0)*+{M_{\textrm{on}}}="3";
(15,0)*+{M_{\textrm{off}}}="4";
{\ar@/^.5pc/^{g} "1";"2"};
{\ar@/^.5pc/^{f} "2";"1"};
\endxy
$ may be useful for picturing the system and its discrete dynamics.
\end{example}

\begin{example}[Two rooms with heaters and thermostats] \label{ex:1.2}
Now imagine that we have a two room house with two heaters and two thermostats.   
 The dynamics now becomes more complicated since each heater is controlled by its own thermostat and the thermostats need not be in sync.  For example, room one may reach 0 first.  Then its heater will turn on. 
Once the first room heats up to temperature 1,  heater 1 will be turned off.   By this time the second room may or may not be at 0.   If the second room is above zero,  it will continue to cool.  Since its temperature is lower than that of the first room it may reach 0 before the temperature in the first room does.  Then its heater will be turned on and so on.    Many other scenarios are also possible. To describe all the possible  dynamics of this system quantitatively we will need to consider the product 
\[
( M^1_{\textrm{on}}\sqcup M^1_{\textrm{off}} ) \times ( M^2_{\textrm{on}}\sqcup M^2_{\textrm{off}}) = 
\bigsqcup_{\alpha, \beta\in \{\textrm{on}, \textrm{off}\}} M^1_\alpha \times M^2 _\beta
\]
On each product $M^1_\alpha \times M^2 _\beta =[0,1]^2$ (which is a manifold with corners, see Appendix~\ref{app:A}) we would have a vector field $X_{\alpha, \beta}$. 
If the two rooms are completely thermally isolated from each other then each $X_{\alpha, \beta} $ is a product of vector fields on the corresponding factors ($M^1_\alpha$ and $M^2_\beta$).  If there is a heat exchange between the rooms  the vector fields $X_{\alpha, \beta} $ would have to be more complicated (recall that a vector field on the product of two manifolds is rarely a product of vector fields on the factors).
 We will also have 12 partial maps between the various products $M^1_\alpha \times M^2 _\beta$.   For example we have a map 
\[
f_1 \times id_{M^2_{\on}}: M^1_\on \times M^2_\on \to M^1_\off \times M^2_\on
\]
defined on $\{1\}\times [0,1)\subset M^1_\on \times M^2_\on$.  It models the fact that when the temperature in room 1 reaches 1 and the temperature in room 2 is below 1, the heater in room 1 turns off while the heater in room 2 keeps going.   The sources and targets of the  partial maps can be pictured by the following graph:
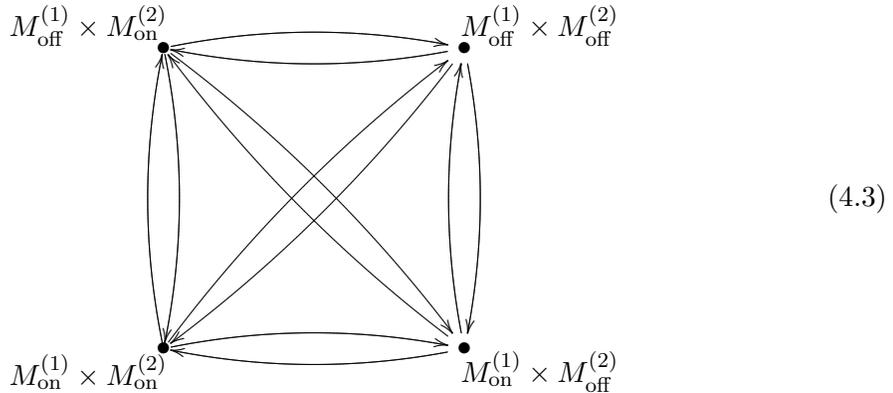
\begin{equation} \label{eq:graph_product}
\xy
(-30,-23)*{M^{(1)}_{\textrm{on}} \times M^{(2)}_{\textrm{on}} }; 
(30,-23)*{M^{(1)}_{\textrm{on}} \times M^{(2)}_{\textrm{off}} }; 
(-30,23)*{M^{(1)}_{\textrm{off}} \times M^{(2)}_{\textrm{on}} }; 
(30,23)*{M^{(1)}_{\textrm{off}} \times M^{(2)}_{\textrm{off}} }; 
(-20,-20)*{\bullet}="1"; 
(-20,20)*{\bullet}="2"; 
(20,20)*+{\bullet}="3";
(20,-20)*+{\bullet}="4";
{\ar@/^.5pc/^{} "1";"2"};
{\ar@/^.5pc/^{} "2";"1"};
{\ar@/^.5pc/^{} "1";"3"};
{\ar@/^.5pc/^{} "3";"1"};
{\ar@/^.5pc/^{} "1";"4"};
{\ar@/^.5pc/^{} "4";"1"};
{\ar@/^.5pc/^{} "2";"3"};
{\ar@/^.5pc/^{} "2";"4"};
{\ar@/^.5pc/^{} "3";"2"};
{\ar@/^.5pc/^{} "3";"4"};
{\ar@/^.5pc/^{} "4";"3"};
{\ar@/^.5pc/^{} "4";"2"};
\endxy
\end{equation}
\end{example}
\mbox{}\\[12pt]

Here are the desiderata for a definition of  a hybrid phase space that
we will introduce in this section.
  \begin{itemize}
\item Any hybrid dynamical systems should be a pair $(a,X)$ where $a$
  is a hybrid phase space and $X$ is a vector field on the manifold
  ``underlying'' $a$.  For example the manifold underlying  hybrid phase space of
  Example~\ref{ex:1.1} should be the disjoint union  $[0,1]\sqcup
  [0,1]$ of two copies of the closed interval $[0,1]$.
  
  \item Hybrid phase spaces should form a category; we denote it by
    $\HyPh$.   The assignment of the underlying manifold to a hybrid
    phase space should be functorial.  That is, there should be a
    functor $\UU: \HyPh\to \Man$ from the category $\HyPh$ of hybrid
    phase spaces to the category of manifolds with corners.

   \item The category $\HyPh$ should have finite products that behave
     ``correctly.''  In particular the hybrid phase space of
     Example~\ref{ex:1.2} should be the product of two copies of the  hybrid phase
     space of Example~\ref{ex:1.1}.

   \item  Recall that the category of continuous time dynamical
     systems is the category of elements of the functor $\scrX: \Man
     \to \RelVect$ (the functor $\scrX$  assigns to a manifold $M$ the space of
     vector fields on $M$).  The category $\HyDS$
       of hybrid dynamical systems should be the category of elements
       of the  composite functor $\scrX\circ \UU: \HyPh\to \RelVect$.
       In particular a hybrid dynamical system should be a pair $(a,
       X)$ where $a$ is a hybrid phase space and $X$ is a vector field
       on the underlying manifold with corners $\UU(a)$.
              
     \item Executions of hybrid dynamical systems should be morphisms
       in the category $\HyDS$, and morphisms of hybrid dynamical
       systems should take executions to executions.    Note that
       since we are dealing with a category of elements, morphisms in
       $\HyDS$ should  morphisms of hybrid phase spaces satisfying
       appropriate conditions.  In particular executions should be
       morphisms of hybrid phase spaces as well.
     \end{itemize}

 Here is a brief explanation of the desirability  of the last item.  It is analogous to the
 following fact about continuous time dynamical systems.  An integral
 curve of a vector field $X$ on a manifold $M$ is a smooth map
 $\gamma:I\to M$, where $I$ is an interval, subject to the condition
 that
 \[
(T\gamma)_s \left( \frac{d}{dt}\right) = X (\gamma(s))
   \]
for all times $s\in I$.  Here as elsewhere in the paper $T\gamma:
TI\to TM$ is the differential of $\gamma$.    We therefore can view an
integral curve of a vector field $X$ as a smooth map from an interval
to the manifold that relates the constant vector field $\frac{d}{dt}$
and the vector field $X$.   Moreover if $f:M\to N$ is a smooth map
between manifolds, $Y$ is a vector field on $N$ which is $f$-related
to a vector field $X$ on $M$ and $\gamma:I\to M$ is an integral curve
of the vector field $X$ then by the chain rule $f\circ \gamma$ is an integral curve of
the vector field $Y$.   We want an analogous result to hold for
hybrid dynamical systems.

Thus an execution of a hybrid dynamical system $(a, X)$ should be a
map from a model ``hybrid time'' dynamical system
$(\scrI, \partial)$ (a hybrid analogue of an interaval $I \subset \R$
with the vector field $\frac{d}{dt}$) to the system $(a, X)$.
     
We construct the category $\HyPh$ of hybrid phase spaces by
categorifying  the category  of lists
  $(\Set/\sfC)^\Rightarrow$: we   replace the category $\Set$ of sets
  with the 2-category $\Cat$ of small categories and $\sfC$ with the
  double category $\RelMan^\Box$ of manifolds with corners, relations
  and smooth maps.  

  To motivate our definition consider a traditional definition of a
  hybrid dynamical system:
  Definition~\ref{def:hds_trad}.  Traditionally a hybrid dynamical
  system  consists of the following data: a directed
  graph $A =\{A_1\toto A_0\}$, an assignment $A_0 \ni y\mapsto R_y$ of
  a manifold with corners to each vertex $y$ of $A$, an assignment
\[
(x\xrightarrow{\gamma}y )\mapsto (R_x \xrightarrow{R_\gamma} R_y)
  \]
  of a relation for each arrow $\gamma$ of $A$ and an assignment of a vector field
  $X_y$ on each manifold $R_y$.  We can view the collection
  $\{X_y\in \scrX(R_y)\}_{y\in A_0}$ of vector fields 
   as a single vector
  field $X$ on the disjoint union $\bigsqcup_{y\in A_0} R_y$.  We can
  view the assignments $y\mapsto R_y$, $\gamma \mapsto R_\gamma$ as a
  single map of graphs.  The source of this a map is the graph $A$.
  The target is the graph $U(\RelMan)$ which is the graph  underlying the category
  $\RelMan$ of manifolds with corners and relations, see
  Remark~\ref{rmrk:RelMan}.   Concretely  the vertices of the graph $U(\RelMan)$
  are manifolds with corners and the arrows are the set-theoretic
  relations. We ignore the 2-category structure of
  $\RelMan$ for the time being.  Thus a hybrid dynamical system is a
  pair
  $(A\xrightarrow{R} U(\RelMan), X\in \scrX(\bigsqcup_{y\in A_0}
  R_y))$.
   It seems reasonable at this point to
  define the phase space of a hybrid dynamical system
  $(A\xrightarrow{R} U(\RelMan), X\in \scrX(\bigsqcup_{y\in A_0} R_a))$
  to be the map of graphs $R:A\to U(\RelMan)$.  Unfortunately
  Examples~\ref{ex:1.1} and \ref{ex:1.2} indicate that this is not
  quite right.  The issue is that the product of the graph
\begin{equation}
  A= \xy
(-10,0)*{\bullet}="1"; 
(10,0)*{\bullet}="2";
{\ar@/^.5pc/ "1";"2"};
{\ar@/^.5pc/ "2";"1"};
\endxy
\end{equation} with itself is the graph 
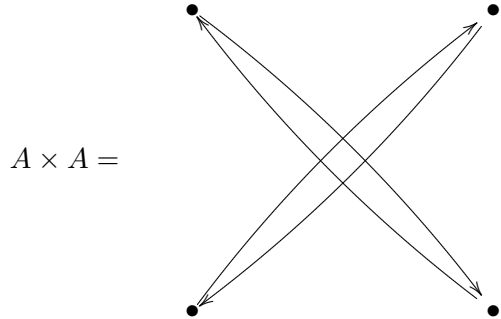
\[ A\times A = \qquad \xy
(-20,-20)*{\bullet}="1"; 
(-20,20)*{\bullet}="2"; 
(20,20)*+{\bullet}="3";
(20,-20)*+{\bullet}="4";
{\ar@/^.5pc/^{} "1";"3"};
{\ar@/^.5pc/^{} "3";"1"};
{\ar@/^.5pc/^{} "2";"4"};
{\ar@/^.5pc/^{} "4";"2"};
\endxy
\]
and not the graph
\begin{equation}
B = \qquad  \xy
(-20,-20)*{\bullet}="1"; 
(-20,20)*{\bullet}="2"; 
(20,20)*+{\bullet}="3";
(20,-20)*+{\bullet}="4";
{\ar@/^.5pc/^{} "1";"2"};
{\ar@/^.5pc/^{} "2";"1"};
{\ar@/^.5pc/^{} "1";"3"};
{\ar@/^.5pc/^{} "3";"1"};
{\ar@/^.5pc/^{} "1";"4"};
{\ar@/^.5pc/^{} "4";"1"};
{\ar@/^.5pc/^{} "2";"3"};
{\ar@/^.5pc/^{} "2";"4"};
{\ar@/^.5pc/^{} "3";"2"};
{\ar@/^.5pc/^{} "3";"4"};
{\ar@/^.5pc/^{} "4";"3"};
{\ar@/^.5pc/^{} "4";"2"};
\endxy \quad 
.
\end{equation}
On the other hand the phase space of the hybrid dynamical system of
Example~\ref{ex:1.2} should be the product two copies of the phase
space of the system in Example~\ref{ex:1.1}.  We choose the following
solution to the problem.  Recall that the forgetful functor
$U:\CAT \to \Graph$ from the category of (large) categories to the
categories of graphs is part of the free/forgetful adjunction
$\Free:\Graph \leftrightarrows \CAT:U$.  It is easy to see that
\[
\Free(A)\times \Free(A) = \Free(B)
\]
for the graphs $A$ and $B$ above.
So we now provisionally (re)define a hybrid phase space to be a functor $R: \Free(\Gamma) \to
\RelMan$ from the free category on some graph $\Gamma$ to the (2-) category
$\RelMan$ of manifolds with corners and relations.

We believe that our solution captures the following idea.  Suppose
that we have a hybrid dynamical system that consists of two
interacting subsystems.  Then its underlying phase spaces is a product
of two hybrid phase spaces which we call $a_1$ and $a_2$.  An execution of
the big system is, in particular, a map $\sigma$ from a hybrid time
interval $\scrI$ to the product $a_1\times a_2$.  Hence $\sigma =
(\sigma_1, \sigma_2): \scrI\to a_1\times a_2$.  It may happen
that $\sigma_1$ is forced to undergo a discrete transition at time $t$
while $\sigma_2$ at $t$ is evolving continuously.  Therefore given an
arrow $\gamma$ in the graph associated with the first subsystem and
the vertex $b$ of the second subsystem the graph of the product system
needs to have an arrow that corresponds to the pair $(\gamma, b)$.
The construction we chose accomplishes exactly that by assigning to
each vertex $b$ of the second graph the identity arrow $\id_b$ and
interpreting the pair $(\gamma, b)$ as the pair of arrow
$(\gamma, \id_b)$ in the product.

But why stick with free categories?
We may as well (again, provisionally) define a hybrid phase space to
be a functor $a:S_a\to \RelMan$ from some small category $S_a$ to the
category $\RelMan$ of manifolds with corners and relations.  Note that
given a functor $a:S_a\to \RelMan$ there is an underling manifold
$\UU(a): = \bigsqcup_{x\in (S_a)_0} a(x)$ where the coproduct is taken
over the set $(S_a)_0$ of objects of the small category $S_a$.
Consequently we can define a hybrid dynamical system to be a pair
$(S_a\xrightarrow{a}\RelMan, X\in \scrX(\bigsqcup_{x\in (S_a)_0}
a(x)))$.  We will revise the definition of a hybrid phase space and of
a hybrid dynamical system one more time in order to obtain a slicker
definition of morphisms between two hybrid phase spaces.  We will then
show that there is a product preserving forgetful functor
$\UU:\HyPh\to \Man$ from the category of hybrid phase spaces to the
category of manifolds with corners.

\begin{remark}
  There are alternatives to the definition of a hybrid phase space as
  a functor from some small category into $\RelMan$.  The first
  alternative is to use reflexive graphs instead of categories.
  Recall that a reflexive graph is a directed graph that in addition to the source
  and target maps $\frs$, $\frt$ from arrows to nodes also has a unit map
  $\fru$ from nodes to arrows.  Moreover the unit map $\fru$ is a section of
  both $\frs$ and of $\frt$.  In other words, every reflexive graph
  assigns to each vertex $b$ an ``identity arrow'' $\id_b$ whose
  source and target are $b$.  Unlike categories reflexive graphs have
  no composition map.  One can show that products of reflexive graphs
  behave the way we would want them to behave in our examples.

Another alternative is to use labelled transition systems as
the source of our map into $\RelMan$.  Unfortunately expressing
parallel composition of labelled transition systems in the category
theoretic language is  awkward; see \cite{WN}. 
 \end{remark} 

 We are now in position to formulate a definition of a  hybrid phase
 space.  We will later reformulate this definition to make it more
 succinct (Definition~\ref{def:hyph}).

\begin{definition}  A {\sf hybrid phase space} is a
  functor $a:S_a\to \RelMan$ from some category $S_a$ to the 1-category
  $\RelMan$ of manifolds with corners and set-theoretic relations.
  
A {\sf  map} from a hybrid phase space $a:S_a\to \RelMan$ to a hybrid
phase space $b:S_b\to \RelMan$ is a functor $\varphi:S_a\to S_b$
together with a collection of smooth maps $\{f_x:a(x)\to
b(\varphi(x))\}_{x\in (S_a)_0}$ so that for each arrow
$x\xrightarrow{\gamma}y$ of the category $S_a$ the diagram 
\[
\xy
(-15,10)*+{a(x)}="1";
(-15,-10)*+{b(\varphi(x))}="2";
(15, 10)*+{a(y)}="3";
(15, -10)*+{b(\varphi(y))}="4";
{\ar@{->}_{f_x} "1";"2"};
{\ar@{->}^{a(\gamma)} "1";"3"};
{\ar@{->}_{b(\varphi(\gamma))} "2";"4"};
{\ar@{->}^{f_y} "3";"4"};
{\ar@{=>}^{} (0,3); (0,-3)};
\endxy 
\]
is a 2-cell in the double category $\RelMan^\Box$.  In other words  we
require that for each $x\xrightarrow{\gamma}y\in (S_a)_1$ the smooth maps
\[
f_x\times f_y:a(x)\times a(y) \to b(\varphi(x))\times b(\varphi(y))
\]
map the relation $a(\gamma) \subset a(x)\times a(y)$ to the relation
$b(\varphi(\gamma))\subset b(\varphi(x))\times b(\varphi(y))$.
 \end{definition} 
It is not hard to show that  maps $(\varphi, f):a\to b$ and
$(\psi, g):b\to c$ of hybrid phase  spaces can be composed.  We set
\[
(\psi, g)\circ (\varphi, f):= (\psi\circ \varphi,
h)
\]
where
\[
h_x := g_{\varphi(x)} \circ f_x
  \]
for every object $x$ of the category $S_a$.  

 Hybrid phase  spaces and their maps form a
category   that we denote by $\HyPh$. 
Moreover the same argument as in Remark~\ref{rmrk:2.43} shows that 
the assignment
\[
(S_a\xrightarrow{a} \RelMan)\mapsto \UU(a):= \bigsqcup _{x\in
  (s_a)_0} a(x)
  \]
  extends to a functor
  \[
\UU: \HyPh \to \Man
    \]
    from the category of hybrid phase spaces to the category of
    manifolds with corners.  The functor $\UU$ forgets the reset
    relations.  We will give another description of the functor
    $\UU:\HyPh \to \Man$ in Remark~\ref{rmrk:4.8}.

In order to make our definition of the category $\HyPh$ of hybrid
phase spaces more compact we  categorify the definition of the category of lists
$(\Set/\sfC)^\Rightarrow$.   Namely we  replace $\sfC$ with a double category
$\DD$ and replacing $\Set$ with the category of small categories
$\Cat$.
\begin{remark}
Let $\D$ be a double category and $\sfC$ an ordinary small
category, that is, an object of $\Cat$.  A strict double functor
$f:\sfC \to \D$ then assigns to each object $x\in \sfC$ the vertical
identity arrow $\id_{f(x)}$ on an object $f(x) \in \D_0$.  To each
arrow $x\xrightarrow{\gamma}y$ of $\sfC$ the functor $f$ assigns the
identity 2-cell $\id_{f(\gamma)}: f(\gamma) \to f(\gamma)$
from the horizontal arrow $f(\gamma):f(x)\to f(y)$ to itself.   Given
two double functors $f,g:\sfC\to \DD$ a vertical transformation
$\alpha:f\Rightarrow g$ assigns to each object $x$ of $\sfC$ a vertical
1-cell $(\alpha_0)_x:f(x)\to g(x)$ and to each arrow
$x\xrightarrow{\gamma} y $ of $\sfC$ a 2-cell
\[
\xy
(-15,10)*+{f(x)}="1";
(-15,-10)*+{g(x)}="2";
(15, 10)*+{f(y))}="3";
(15, -10)*+{g(y)}="4";
{\ar@{->}_{(\alpha_0)_x} "1";"2"};
{\ar@{->}^{f(\gamma)} "1";"3"};
{\ar@{->}_{g(\gamma)} "2";"4"};
{\ar@{->}^{(\alpha_0)_y} "3";"4"};
{\ar@{=>}^{(\alpha_1)_\gamma} (0,3); (0,-3)};
\endxy
\]
\end{remark}
We now record a generalization of Definition~\ref{def:2.40}.

\begin{definition}[The categorified category of lists $(\Cat/\DD)^\Rightarrow$]
\label{def:3.double}
  Fix a double category $\DD$.  An object of the categorified category of lists
  $(\Cat/\DD)^\Rightarrow $ is a double functor $\tau:\sfC \to \DD$
  where $\sfC$ is a small category  thought of
  as a discrete double  category.  A morphism in $(\sfC/\DD)^\Rightarrow$ from a double functor
  $\tau:\sfC\to \DD$ to a double functor $\tau':\sfC'\to \DD$ is a
  commuting triangle of the form
  \[
\xy
(-10, 10)*+{\sfC} ="1"; 
(10, 10)*+{\sfC'} ="2";
(0,-2)*+{\DD }="3";
{\ar@{->}_{\tau} "1";"3"};
{\ar@{->}^{\varphi} "1";"2"};
{\ar@{->}^{\tau'} "2";"3"};
{\ar@{<=}_{\scriptstyle \Phi} (4,6)*{};(-0.4,4)*{}} ; 
\endxy
    \]
where 
  $\varphi:\sfC\to \sfC'$ is a double functor and $\Phi: \tau \Rightarrow \tau'
  \circ \varphi$ is a vertical transformation.
    Given a pair of
  composable morphisms $(\varphi, \Phi): (\sfC\xrightarrow{\tau}\DD)\to
  (\sfC'\xrightarrow{\tau'} \DD)$ and $(\psi, \Psi) : (\sfC'\xrightarrow{\tau'} \DD)\to 
  (\sfC''\xrightarrow{\tau''} \DD)$ their composite is defined by
  pasting of the 2-commuting triangles.  That is, the composite is the pair $(\psi
  \circ \varphi, \Xi)$ where 
\[
\Xi = (\Psi \circ \varphi) \circ _v \Phi.
\] 
 Here $\Psi \circ \varphi: \tau \Rightarrow \tau''$ is the whiskering
of the vertical transformation with a double functor and $\circ_v$ denotes the
vertical composition of vertical transformations (Remark~\ref{rmrk:2.29}).
\end{definition}

\begin{definition}[A category $\HyPh$ of hybrid phase
  spaces] \label{def:hyph}
We define a    category $\HyPh$ {\sf of hybrid phase spaces}  to be
the categorified category of lists $(\Cat/\RelMan^\Box)^\Rightarrow$
(see Definition~\ref{def:3.double}) where $\Cat$ is the category of
small categories and $\RelMan^\Box$ is the double category of
manifolds with corners, smooth maps and arbitrary relations
(Definition~\ref{def:relmanbox}).  Thus
\[
\HyPh:= (\Cat/\RelMan^\Box)^\Rightarrow.
\]
More concretely the objects of $\HyPh$ are double functors $a:S_a\to
\RelMan^\Box$ (where the small category $S_a$ is considered  a
discrete double category).  A morphism from a hybrid phase space $a$
to a hybrid phase space $b$  is a
pair $(\varphi, \Phi)$ where $\varphi: S_a\to S_b$ is a functor between
two discrete double categories (i.e., an ordinary functor) and $\Phi:
a\Rightarrow b\circ \varphi$ is a vertical transformation.
\end{definition}

\begin{remark}
  Unless noted otherwise we assume that our hybrid phase spaces are
  ``nonempty'': given a hybrid phase space $a:S_a\to \RelMan^\Box$ we
  assume that the category $S_a$ is nonempty and that for every object
  $x$ of $S_a$ the manifold with corners $a(x)$ is also
  nonempty. \end{remark}

\begin{remark}\label{rmrk:4.8}
We now give another description of the functor $\UU: \HyPh\to \Man$.
Recall that there is a forgetful functor $U:\Cat\to \Set$.   To a
small category $\sfC$ the functor $U$ assigns its set $\sfC_0$ of objects. To a 
functor $\sfC \xrightarrow{f} \sfD$ between two small categories it
assigns the map $\sfC_0\xrightarrow{f_0} \sfD_0$ on objects.  Note
next that if $a:S_a\to \RelMan^\Box$ is a double functor from a small
category $S_a$ then the $a_0$ component of $a$ is an unordered list of
manifolds $a_0: (S_a)_0 \to \Man$.  Hence we can view
$a_0$ as an object of the category $(\Set/\Man)^\Rightarrow$.
Given a morphism $(\varphi, \Phi): a\to b$ in $(\Cat/\RelMan^\Box)$
the pair   $(\varphi_0, \Phi_0)$ is a morphism in
$(\Set/\Man)^\Rightarrow$ from $a_0$ to $b_0$.  It is easy to see that
this gives us the forgetful functor $(\Cat/\RelMan^\Box)  \to
(\Set/\Man)^\Rightarrow $.   We now compose this forgetful functor
with the canonical functor $\rotpi: (\Set/\Man)^\Rightarrow \to \Man$
and obtain the desired functor $\UU: (\Cat/\RelMan^\Box)^\Rightarrow\to \Man$.
\end{remark}

\begin{proposition}
The category $\HyPh$ of hybrid phase spaces has finite products.
\end{proposition}

\begin{proof} The proof uses the fact that the categories $\Cat$ of
  small categories and functors, and $\Man$ of manifolds with corners
  and smooth maps  have finite
  products.  Additionally the category $\RelMan$ of manifolds and set-theoretic
  relations has a monoidal product which on objects is just the
  product of manifolds with corners. 
  
Let $[0]$  denote a category with one object and one morphism.
A terminal object in the category $\HyPh$ of hybrid phase spaces is
a functor $*: [0] \to \RelMan^\Box$ that assigns to the one
object 0 of 
$[0]$ a one point manifold $\{\bullet\}$.  It assigns to the identity
arrow $\id_0$ of $[0]$ the identity relation $\{(\bullet,
\bullet)\}$.  We claim that for any  hybrid phase space $a:S_a\to
\RelMan^\Box$ there is a unique morphism $(\varphi, \Phi):a \to *$ of
hybrid phase spaces.  Since $[0]$ is terminal in $\Cat$, there is a
unique functor $\varphi:S_a\to [0]$.  Since one point manifolds are
terminal in $\Man$ for every object $x$ of $S_a$ there is a unique
smooth map $\Phi_x: a(x) \to \{\bullet\}$.    The maps
$\{\Phi_x\}_{x\in S_a}$ are components of a unique vertical
transformation $\Phi: a\Rightarrow *\circ \varphi$.

Recall that the category $\RelMan$ of manifolds with corners and
relations has a monoidal product: given a relation $R:M\to N$ and a
relation $S:Q\to P$ their product $R\times S \subset (M\times Q) \times (N
\times P)$ is the set
\begin{equation}\label{eq:mon-pr-rel}
R\times S := \{(m,q, n,p) \in  (M\times Q) \times (N
\times P)\mid (m,n) \in R, (q,p) \in S\}
  \end{equation}

 Suppose now that $a:S_a\to \RelMan^\Box$, $b:S_b\to
\RelMan^\Box$ are two hybrid phase  spaces.   Their product $a\times
b$ is a functor from the product category $S_a\times
S_b$.  We define it as follows. Given
an object $(x,y) \in S_a\times S_b$ we set
\[
  (a\times b) (x,y) := a(x)\times b(y),
\]
where $a(x)\times b(y)$ is the product of two manifold with corners
(the categorical product in $\Man$).
Given an arrow $(x\xrightarrow{\gamma} x', y \xrightarrow{\nu} y')\in
S_a \times S_b$ we set
\[
  (a\times b) (\gamma, \nu):= a(\gamma) \times b(\nu),
\]
where  $a(\gamma) \times b(\nu)$ is the (monoidal) product of two relations given
by \eqref{eq:mon-pr-rel}.   The two projections $\pi_1= (\varpi_1, p_1):a\times
b \to a$, $\pi_2: (\varpi_2, p_2):a\times
b \to b$ are constructed  as follows. Since $\Cat$ has binary products
we have the canonical projection functor $\varpi_1: S_a\times S_b\to S_a$.
Since $\Man$ has binary products, for any pair of objects $(x,y)$
in $S_a\times S_b$ we have the canonical projection $(p_1)_{(x,y)}:
a(x)\times b(y)\to a(x)$.  These projections assemble into a vertical
transformation $p_1: a\times b \Rightarrow a \circ \varpi_1$ between
double functors.  The construction of $\pi_2 = (\varpi_2, p_2):a\times
b \to b$ is the same.

It remains to check  that $a\times b$ with the two projections
satisfies the universal property of product in $\sF{HyPh}$, i.e.\ that
given any pair of morphisms of hybrid phase spaces $z_1= (\zeta_1, Z_1):c\rightarrow
a$ and $z_2= (\zeta_2, Z_2):c\rightarrow b$, there is unique morphism
$z= (\zeta, Z):c\dashrightarrow a\times b$ so that the diagram
\[
  \begin{tikzcd} c\arrow[drr,"z_2"]\arrow[dr,dashed]\arrow[ddr,swap,"z_1"] & & \\
& a\times b\arrow[d,"\pi_1"] \arrow[r,swap,"\pi_2"] & 	b\\
& a &
\end{tikzcd}
\]
commutes.  Since $\Cat$ has products, there is a unique functor
$\zeta:S_c\to S_{a\times b} = S_a \times S_b$ so that $\varpi_1 \circ
\zeta = \zeta_1$ and $\varpi_2 \circ
\zeta = \zeta_2$.    For every object $x$ of $S_c$ we have two smooth
maps 
\[
(Z_1)_x: c(x)\to a (\zeta_1(x))\qquad \textrm{and} \qquad (Z_2)_x: c(x)\to b (\zeta_2(x)).
\]
Since the category $\Man$ of manifolds with corners and smooth maps
has binary products there exists a unique smooth map
\[
Z_x : c(x) \to a (\zeta_1(x)) \times b (\zeta_2(x))
\]
so that 
\[
(Z_1)_x = (p_1)_x \qquad \textrm{and} \qquad (Z_2)_x = (p_2)_x. 
\]
The collection of maps $\{Z_x\}_{x\in S_c}$ define a vertical
transformation $Z:c\Rightarrow (a\times b) \circ \zeta$.

Finally since $\HyPh$ has a terminal object and binary products it has
finite products.
 \end{proof} 

\begin{proposition}
The forgetful functor $\UU: \HyPh\to \Man$ from the category of hybrid
phase spaces to the category of manifolds with corners preserves finite products.
\end{proposition}

\begin{proof}
  Clearly $\UU$ takes the terminal object $*: \one  \to \RelMan^\Box$ in
  the category $\HyPh$ to a one point manifold $*$, which is terminal
  in the category $\Man$.

  It remains to check that $\UU$ preserves binary products.  Namely we
  check that for any two hybrid phase spaces $a,b\in \HyPh$ the
  manifolds with corners $\UU(a)\times \UU(b) $ and $\UU(a\times b)$
  are canonically diffeomorphic.  This amounts to checking that for
  any two lists $\mu:X\to \Man$ and $\nu:Y\to \Man$ in $\Set/\Man$ the
  manifolds $\bigsqcup_{(x,y)\in X\times Y} \mu(x)\times \nu(y)$ and
  $\left(\bigsqcup _{x\in X}\mu (x) \right) \times \left(
    \bigsqcup_{y\in Y} \nu(y) \right)$ are canonically diffeomorphic.

Let $\sfC$ be a category with coproducts and binary products.   Given
two objects  $\mu: X\to \sfC$ and $\nu: Y\to \sfC $ of the category of
lists $\Set/\sfC$ we have a canonical map
\[
\scrP: \bigsqcup_{(x,y)\in X\times Y} \mu(x)\times \nu(y) \to
\left(\bigsqcup _{x\in X}\mu (x) 
\right) \times \left(    \bigsqcup_{y\in Y} \nu(y)  \right).
 \] 
 It is induced by the products of the structure maps
\[
 \mu(x_0)\times \nu(y_0) \xrightarrow{
   \imath_{x_0}\times \imath_{y_0}}
 \left(\bigsqcup _{x\in X}\mu (x) 
\right) \times \left(    \bigsqcup_{y\in Y} \nu(y)  \right).
\]
It is well known that in the category $\Set$ of sets the map $\scrP$
is a bijection (see \cite[Proposition~8.6]{Awodey}).  In the category $\Man$ of manifolds
with corners the structure maps $\imath_{x_0} :\mu(x_0)\to \bigsqcup
_{x\in X} \mu(x)$ are open embeddings.   Consequently the products $\mu(x_0)\times \nu(y_0) \xrightarrow{
   \imath_{x_0}\times \imath_{y_0}}
 \left(\bigsqcup _{x\in X}\mu (x) 
\right) \times \left(    \bigsqcup_{y\in Y} \nu(y)  \right)$ are open
embeddings as well and, in particular, are local diffeomorphisms.
Consequently in the category $\Man$ the map $\scrP$ is a local diffeomorphism and a
bijection of the underlying sets, hence a diffeomorphism.
\end{proof}

\begin{lemma}
  Suppose $\sfC\xrightarrow{F}\sfD$ is a functor.  Then $F$ induces a
  functor $F_*: (\FinSet/\sfC)^\Leftarrow \to
  (\FinSet/\sfD)^\Leftarrow$ which is given by
  \[
    F_* \left(
      \xy
(-10, 6)*+{X} ="1"; 
(10, 6)*+{Y} ="2";
(0,-6)*+{\sfC }="3";
{\ar@{->}_{\tau} "1";"3"};
{\ar@{->}^{\varphi} "1";"2"};
{\ar@{->}^{\mu} "2";"3"};
{\ar@{=>}_{\scriptstyle \Phi} (4,2)*{};(-0.4,0)*{}} ; 
\endxy 
\right)  =
\xy
(-10, 6)*+{X} ="1"; 
(10, 6)*+{Y} ="2";
(0,-8)*+{\sfD }="3";
{\ar@{->}_{F\circ \tau} "1";"3"};
{\ar@{->}^{\varphi} "1";"2"};
{\ar@{->}^{F\circ \mu} "2";"3"};
{\ar@{=>}_{\scriptstyle F\circ \Phi} (4,2)*{};(-0.4,-1)*{}} ; 
\endxy \quad.
    \]
 If $\sfC$ and $\sfD$ have finite products and $F$ is product
 preserving then the diagram
 \[
   \xy
(-15,10 )*+{(\FinSet/\sfC)^\Leftarrow} ="1"; 
(15,10 )*+{\sfC} ="2";
(-15,-10 )*+{(\FinSet/\sfD)^\Leftarrow} ="3";
(15,-10 )*+{\sfD} ="4";
{\ar@{->}^{\qquad \Pi} "1";"2"};
{\ar@{->}_{F_*} "1";"3"};
{\ar@{->}^{F} "2";"4"};
{\ar@{->}_{\qquad \Pi} "3";"4"};
   \endxy
 \]
 commutes (up to a unique isomorphism).
\end{lemma}  

\begin{proof}
  Omitted.
\end{proof}

\begin{corollary}\label{cor:4.12}
  The diagram
  \[
   \xy
(-15,10 )*+{(\FinSet/\HyPh)^\Leftarrow} ="1"; 
(15,10 )*+{\HyPh} ="2";
(-15,-10 )*+{(\FinSet/\Man)^\Leftarrow} ="3";
(15,-10 )*+{\Man} ="4";
{\ar@{->}^{\qquad \Pi} "1";"2"};
{\ar@{->}_{\UU_*} "1";"3"};
{\ar@{->}^{\UU} "2";"4"};
{\ar@{->}_{\qquad \Pi} "3";"4"};
   \endxy
 \]
 commutes (up to a unique isomorphism).
\end{corollary}

\begin{example}[Compare with Example~\ref{ex:exPiphi2}] \label{ex:exPiphi3}
  Let $\sfC =\HyPh$, the category of hybrid phase spaces, $X = Y = \{1,2, 3\}$.
  Fix two hybrid phase spaces $A$ and $B$ and  a  map $s:A\to B$ of
  hybrid phase spaces.  Let $\tau, \mu: X, Y\to \HyPh$ be
  the constant maps defined by $\tau(j) = B$, $\mu(j) =A$ for all
  $j$.    We define a morphism $(\varphi,
  \Phi): \tau \to \mu$ in $(\FinSet/\HyPh)^\Leftarrow$ as follows.
As in Example~\ref{ex:exPiphi2} we define $\varphi$ by 
  \[
\varphi(1) = 2,\qquad \varphi(2) = 1, \qquad \varphi(3) =2.
\]
We define $\Phi_i: \mu (\varphi(i)) = A \to \tau(i) = B$ to be the map
$s:A \to B$ for all $i$.
  Then $(\varphi,\Phi): \tau \to \mu$ is a morphism in
  $(\FinSet/\HyPh)^\Leftarrow$.
 By Remark~\ref{rmrk:2.41} we have a map of hybrid phase spaces
 $\Pi (\varphi, \Phi):\Pi(\mu) = A^3 \to B^3 = \Pi(\tau)$.
By Corollary~\ref{cor:4.12} the map $\UU(\Pi(\varphi, \Phi)): \UU(A^3)
= \UU(A)^3\to \UU(B)^3 = \UU (B^3)$ 
 is given by
\[
\UU(\Pi(\varphi,\Phi))\, (a_1, a_2, a_3) = (\UU(s)\,(a_2), \UU(s)\,(a_1), \UU(s)\,( a_2))
\]
for all $(a_1,a_2, a_3) \in (\UU(A))^3$.
\end{example}
 
\subsection{A category $\HDS$ of hybrid dynamical systems}\mbox{}\\
We  define a category $\HDS$ of hybrid dynamical systems as the
category of elements of the functor $\scrX\circ \UU: \HyPh\to
\RelVect$.  Thus we formally record:

\begin{definition}[The category $\HDS$ of hybrid dynamical systems]
The category $\HDS$ of hybrid dynamical systems is the category of
elements of the functor $\scrX \circ \UU : \HyPh \to \RelVect$.
Explicitly an object of the category $\HDS$, that is, a {\sf hybrid
dynamical system},  is a pair $(a, X)$ where $a$ is a hybrid phase space
and $X$ is a vector field on the underlying manifold $\UU(a)$.  A
map of hybrid dynamical systems from $(a,X)$ to $(b,Y)$ is a map
$(\varphi, \Phi):a\to b$ of hybrid phase spaces such that the vector
field $Y$ is $\UU(\varphi,\Phi)$-related to $X$.
\end{definition}

To define executions of a hybrid dynamical system $(a,X)$ we need to
define the hybrid analogue of the continuous time dynamical system
$(I, \frac{d}{dt})$ where  $I$ is an interval and $\frac{d}{dt}$ is the constant vector
field.

\begin{definition}[The hybrid dynamical system $(\scrI, \partial) = (\scrI(\{t_i\}), \partial)  $
  associated with increasing sequence $\{t_i\}_{i=0}^\infty$ of real
  numbers]
  \label{def:4.15}
Let $\{t_i\}$ be an increasing sequence of real numbers.  Let $\scrT$
be the graph with the set of vertices $\scrT_0 = \N$, the set of
natural numbers (that includes 0), the set of arrows $\scrT_1 = \N$
and the source and target maps $\frs \times \frt : \scrT_1\to \scrT_0
\times \scrT_0$ given by $i\mapsto (i, i+1)$.  In other words arrows
of the graph $\scrT$ are of the form $i\xrightarrow{i} i+1$.  Define a map of
graphs $\bar{\tau}: \scrT\to U(\RelMan)$ (recall that $U(\RelMan)$ is the
graph underlying the category of manifolds with corners and relations)
by
\[
  \bar{\tau} (i\xrightarrow{i}i+1) = [t_i,t_{i+1}]\xrightarrow{ \{(t_{i+1},
    t_{i+1})\}} [t_{i+1}, t_{i+2}].
  \]
Here $\{(t_{i+1},t_{i+1})\} \subset [t_i, t_{i+1}]\times [t_{i+1},
t_{i+2}]$ is a relation consisting of one point.  We define the hybrid
phase space $\scrI$ to be the corresponding functor $\tau:
\Free(\scrT)\to \RelMan$.   We define the vector field $\partial$ on
the underlying manifold $\UU(\tau)  = \bigsqcup_{i=0}^\infty [t_i,
t_{i+1}]$ by $\partial|_{[t_i,t_{i+1}]} := \frac{d}{dt}$.
\end{definition}

\begin{definition}\label{def:execution} \label{def:exec}
An {\sf execution} of a hybrid dynamical system $(a, X)$ is an  increasing sequence $\{t_i\}_{i=0}^\infty$ of real
  numbers and a map of
hybrid dynamical systems $(\varphi, \sigma): (\scrI(\{t_i\}), \partial) \to
(a,X)$. Here  $(\scrI(\{t_i\}), \partial) $ is the hybrid dynamical
system associated to the sequence $\{t_i\}_{i=0}^\infty$ (Definition~\ref{def:4.15}).
\end{definition}

\begin{remark}\label{rmrk:exec}
  Unpacking Definition~\ref{def:execution} we see that an execution
  $(\varphi, \sigma)$ of a hybrid dynamical system $(a, X)$ associates
   to each integer $i\geq 0$ an integral curve
  $\sigma_i:[t_i, t_{i+1}]\to a(\varphi_0(i))$ of the vector field
  $X_{\varphi_0(i)}$ on the manifold with corners $ a(\varphi_0(i))$
  so that $(\sigma_i(t_{i+1}), \sigma_{i+1} (t_{i+1}))$ lies in the
  reset relations
  $a(\varphi_1(i\xrightarrow{i}i+1) \subset a(\varphi_0(i))\times
  a(\varphi_0(i+1))$.  Thus Definition~\ref{def:exec} of executions as
  maps of hybrid dynamical systems agrees  with a traditional definition of
  an execution, Definition~\ref{def:chs_execution}.
\end{remark}

\begin{theorem} \label{thm:exec}
A map  $(\psi, \Psi):(a, X)\to (b, Y)$ of hybrid dynamical system
sends the executions of $(a,X)$ to the executions of $(b, Y)$.
 \end{theorem} 

 \begin{proof}
Let $(\varphi, \sigma): (\scrI, \partial) \to (a,X)$ be an execution
of the hybrid dynamical system $(a,X)$.   Then, by definition,
$(\varphi, \sigma)$ is a
map of hybrid dynamical systems. Hence $(\psi, \Psi)\circ (\varphi,
\sigma): (\scrI, \partial) \to (b, Y)$ is also a map of hybrid
dynamical systems, and therefore  an execution.
 \end{proof}

\section{The double category $\HySSub^\Box$ of hybrid surjective
  submersions and hybrid open systems}

We start with some category-theoretic generalities that will help us
to define the category of hybrid surjective submersions $\HySSub$ and
a forgetful product-preserving functor $\UU:\HySSub\to \SSub$ from the
category of hybrid surjective submersions to the
category of surjective submersions.

Suppose $\scC$ is
a category with finite products.  Denote by $\two$ the category with
two objects $0, 1$ and one nonidentity morphism $0\to 1$.   The
functor category $\scC^{\two}$ is the category of arrows of the
category $\scC$.  The objects of $\scC^{\two}$ can be identified with
morphism $c\xrightarrow{f}c'$ of $\scC$.  Under this identification a
morphism in $\scC^{\two}$ from $c\xrightarrow{f} c'$ to
$d\xrightarrow{g}d'$ is the commutative diagram
\[
    \xy
(-8, 6)*+{c'} ="1"; 
(8, 6)*+{c} ="2";
(-8,-6)*+{d'}="3";
(8, -6)*+{d}="4";
{\ar@{->}_{f} "2";"1"};
{\ar@{->}^{g} "4";"3"};
{\ar@{->}_{h} "1";"3"};
{\ar@{->}^{k} "2";"4"};
\endxy  .
\]
If $\scC$ has finite products, then so does the arrow category
$\scC^{\two}$: the terminal object of $\scC^{\two}$ is the unique arrow
$*\to *$ where $*$ is the terminal object of $\scC$.   The binary
product of $c\xrightarrow{f}c'$ and $d\xrightarrow{g}d'$ is
$c\times d \xrightarrow{f\times g} c'\times d'$ and so on.
If $\scD$ is another category with finite products and $\UU:\scC\to
\scD$ is a product -preserving functor, then the induced functor
\[
\scC^{\two} \to \scC^{\two}, \qquad (c\xrightarrow{f} c') \mapsto
(\UU(c)\xrightarrow{\UU(f)} \UU(c') )
  \]
is also product preserving.  We  denote this induced functor again by
$\UU$ and trust that it will not cause any confusion.

\begin{definition}[The category $\HySSub$ of hybrid submersion] \label{def:hybssub}
Observe that the category $\SSub$ of surjective submersions in a full subcategory
of the category of arrows $\Man^{\two}$.

We define the category
$\HySSub$ of {\sf hybrid surjective submersions} to be the preimage
$\UU\inv (\SSub)$ in the category of arrows $\HyPh^{\two}$.   Here
$\UU:\HyPh^{\two}\to \Man^{\two}$ is the forgetful functor induced by
$\UU:\HyPh\to \Man$.

More concretely the objects of $\HySSub$ are   {\sf hybrid surjective
  submersions}.  These  are maps of hybrid phase
spaces $a_\tot \xrightarrow{p_a} a_\st$ such that the underlying maps
of manifolds $\UU(p_a): \UU (a_\tot)\to \UU(a_\st)$ are  surjective
submersions.  A morphism in $\HySSub$ from $a_\tot \xrightarrow{p_a}
a_\st$ to $b_\tot\xrightarrow{p_b} b_\st$ is a pair of maps of hybrid
phase spaces $f_\tot:a_\tot \to b_\tot$, $f_\st:a_\st\to b_\st$ making
the appropriate diagram commute:
\[
f_\st \circ p_a = p_b \circ f_\tot  .
  \]
\end{definition}

\begin{remark} \label{rmrk:4.2} Since $\UU: \HyPh\to \Man$ is product
  preserving, so is the induced forgetful functor
  $\UU: \HyPh^{\two} \to \Man^{\two}$ on arrow categories.  By
  definition $\UU$ takes the category $\HySSub \subset \HyPh^{\two}$ to
  the category $\SSub$ of surjective submersions.  Consequently the
  functor
\begin{equation}\label{eq:UU}
\UU:\HySSub \to \SSub,
\end{equation}
which is a restriction of the forgetful functor $\UU:\HyPh^{[1]} \to \Man^{\two}$,
preserves finite products.
\end{remark}

\begin{remark} \label{rmrk:5.4}
Given two manifold $M, N$ the projections $\pi_1:M\times N\to M$ and
$\pi_2:M\times N\to N$ are surjective submersions.  Since the
forgetful functor 
$\UU:\HyPh\to \Man$ is finite product preserving, for any 
two hybrid phase spaces $a$ and $b$ the images $\UU(\pi_1) :
\UU(a)\times \UU(a) \to  \UU(b)$ and $\UU(\pi_2): \UU(a)\times
\UU(b)\to \UU(b)$ of the canonical
projections $\pi_1:a\times b \to a$ and $\pi_2:a\times b\to b$ are the
canonical projections of the product of manifolds, hence surjective
submersions.  We conclude that the canonical  projections
$\pi_1:a\times b \to a$ and $\pi_2:a\times b\to b$ from the
categorical product in $\HyPh$ to its factors  are
hybrid surjective submersions in the sense of Definition~\ref{def:hybssub}.
\end{remark}

Analogously to the category $\OS$ of open systems
(Example~\ref{ex:3.4}) we have the category $\HyOS$ of hybrid open
systems.  It is defined as follows.
\begin{definition}[The category $\HyOS$ of hybrid open
  systems] \label{def:HyOS} We define the  category $\HyOS$ of {\sf  hybrid open
  systems} to be  the category of elements
  (Definition~\ref{def:cat_of_elements}) of the functor
  $\Crl \circ \UU: \HySSub \to \RelVect$.  Here as before
  $\Crl: \SSub \to \RelVect$ is the functor that assigns to each
  surjective submersion the vector space of open systems on the
  submersions (see Example~\ref{ex:3.4}) and $\UU: \HySSub \to \SSub$
  is the forgetful functor from the category of hybrid surjective
  submersions to surjective submersions (Remark~\ref{rmrk:4.2}).

  More explicitly a {\sf hybrid open system} is a pair $(a, F)$ where
  $a = (a_\tot\xrightarrow{p_a} a_\st)$ is a hybrid surjective
  submersion and $F:\UU(a_\tot)\to T\UU(a_\st)$ is an open system on
  the underlying surjective submersion $\UU(a) = (\UU(a_\tot)\xrightarrow{\UU(p_a)} \UU(a_\st))$.
 \end{definition} 

 Recall that the category $\SSub$ of surjective submersions has a
 special class of morphisms: the interconnection maps
 (Definition~\ref{def:ssub}).  Recall also that together the two types
 of morphisms --- interconnections and ``ordinary'' maps of
 submersions --- define the double category $\SSub ^\Box$ of
 surjective submersions (Definition~\ref{def:ssub_box}).  The
 forgetful functor $\UU:\HySSub\to \SSub$ allows us to transfer this
 double category structure to hybrid submersions.

 \begin{definition}\label{def:hy_int}
A map $f= (f_\tot, f_\st):a\to b$ of hybrid surjective submersions is
an {\sf interconnection morphism} if the underlying map $\UU(f_\st):
\UU(a_\st)\to \UU(b_\st)$  is an interconnection map of surjective submersions.
 \end{definition}  
\begin{example}
Let $a$, $b$ be two hybrid phase spaces and $s:a\to b$ a map of hybrid
phase spaces.  By Remark~\ref{rmrk:5.4} the canonical projection
$\pi_1:a\times b \to a$ is a hybrid surjective submersion.  The
identity map $\id:a\to a$ is also a hybrid surjective submersion.
The diagram
\[
    \xy
(-8, 6)*+{a} ="1"; 
(8, 6)*+{a\times b} ="2";
(-8,-6)*+{a}="3";
(8, -6)*+{a}="4";
{\ar@{->}^{(\id,s)} "1";"2"};
{\ar@{->}^{\pi_1} "2";"4"};
{\ar@{->}_{\id} "1";"3"};
{\ar@{->}^{\id} "3";"4"};
\endxy
\]
commutes.  Hence $\varphi:= ((\id, s), \id)): (a\xrightarrow{\id}a) \to (a\times
b\xrightarrow{\pi_1} b))$ is an interconnection map of hybrid
surjective submersions.

It follows that for any hybrid open system of the form  $(a\times b\xrightarrow{\pi_1} a, F \in \Crl
\UU(a\times b\xrightarrow{\pi_1} a))$ the pair $(a, \UU(\varphi)^* F)$
is a hybrid dynamical system (since $\UU(\varphi)^*F$ is a vector field).
\end{example}

\begin{remark}
Suppose $\varphi: a \to b$ is an interconnection morphism between two
hybrid surjective submersions and $F\in \Crl (\UU(b))$ is a
(continuous time) open system.  Then $(b, F)$ is a hybrid open system
and so is $(a, \UU(\varphi)^*F)$.
 \end{remark} 

\begin{example} \label{ex:hds_on_product} \label{ex:5.9}
In this example we  generalize Example~\ref{rmrk:vf_on_product} from
manifolds to hybrid phase spaces.  
We argue that any hybrid dynamical system of the form $(a\times b, X\in
\scrX(\UU(a\times b)))$ ($a,b$ are hybrid phase spaces) is the result
of interconnection of two hybrid open systems.

As in 
Example~\ref{rmrk:vf_on_product}  the vector field $X = (X_1, X_2) $
where $X_1 :\UU(a\times b)\to T\UU (a)$, $X_2:\UU(a\times b)\to
T\UU(b)$ are open systems.   Here we used the fact that $T\UU(a\times
b) = T(\UU(a) \times \UU(b))$  since the forgetful functor $\UU$
preserves products.     Again, since $\UU$ preserves products and
since the canonical projections $\pr_1:\UU(a)\times \UU(b)\to \UU(a)$,
$\pr_2:\UU(a)\times \UU(b) \to \UU(b)$ are surjective submersions, the
structure 
maps $\pi_1:a\times b\to a$, $\pi_2:a\times b\to b$ are hybrid
surjective submersions.    Denote by $\Delta$ the diagonal map
\[
a\times b \to (a\times b)\times (a\times b).
\]
The diagram of hybrid phases spaces and their maps
\[
\xy
(-15, 6)*+{a\times b} ="1"; 
(15, 6)*+{(a\times b)\times (a\times b)} ="2";
(-15,-6)*+{a\times b}="3";
(15, -6)*+{a\times b}="4";
{\ar@{->}^{\Delta\qquad } "1";"2"};
{\ar@{->}_{\id } "3";"4"};
{\ar@{->}_{\id} "1";"3"};
{\ar@{->}^{\pi_1\times \pi_2} "2";"4"};
\endxy
\]
commutes.   Hence
\[
\psi := (\Delta, \id) :(a\times b\xrightarrow{\id}a\times
b)\to ((a\times b \xrightarrow{\pi_1} a)\times (a\times
b)\xrightarrow{\pi_2} b)
\]
is an interconnection map between two hybrid surjective submersions.
Since $\UU$ is a product preserving functor, $\UU(\Delta):\UU(a\times
b) \to \UU((a\times b)\times (a\times b)) = \UU(a\times b) \times
\UU(a\times b)$ is also the diagonal map.   We have seen in
Example~\ref{rmrk:vf_on_product} that $X = \varphi^* (X_1\times X_2)$
where $\varphi = \UU(\psi)$.   Therefore the hybrid dynamical system
$(a\times b, X)$ is obtained by interconnecting two hybrid open
systems $(a\times b \to a, X_1)$ and $(a\times b\to b, X_2)$.
\end{example} 

\begin{definition}[The double category $\HySSub^\Box$ of
 hybrid surjective  submersions]\label{def:ssub_box}
  The objects of the double category $\HySSub^\Box$ are hybrid surjective
  submersions.  The horizontal 1-morphisms are morphisms of surjective
  submersions.The vertical
  1-morphisms are interconnection morphisms (Definition~\ref{def:hy_int}). 
    The 2-cells of $\SSub^\Box$ are commuting squares 
\[
  \xy
(-10,10)*+{c}="1";
(10, 10)*+{a}="2";
(-10,-5)*+{d}="3";
(10, -5)*+{b}="4";
{\ar@{<-}^{\mu} "1";"2"};
{\ar@{->}_{g} "1";"3"};
{\ar@{->}^{f} "2";"4"};
{\ar@{<-}_{\nu} "3";"4"};
  \endxy
\]
in the category $\HySSub$ of hybrid surjective submersions, where
$\mu,\nu$ are maps of hybrid submersions and $f$, $g$ are the
interconnection morphisms.  In particular the category of objects of
the double category $\HySSub^\Box$ is the category $\HySSub^\inter$ of
hybrid surjective submersions and interconnection morphisms.
\end{definition}

\begin{remark}
It is easy to check that  we have a forgetful functor from the double
category $\HySSub^\Box$ to the double category $\SSub^\Box$.  Again we
denote it by $\UU$.  Thus $\UU:\HySSub^\Box \to \SSub^\Box$.
\end{remark}

\section{Networks of hybrid open systems}

We start this section by reviewing the results of \cite{L2}  on
networks of open systems.  We then generalize the results of \cite{L2}
to networks of {\em hybrid} open systems, which is the main point of
this paper. 
Recall the definition of a network of open systems from \cite{L2}:
\begin{definition} \label{def:network_os}
  A {\sf network of open systems} is a pair $(X\xrightarrow{\tau}
  \SSub, b\xrightarrow{\psi}  \Pi(\tau)  )$ where $\tau$ is an element
  of $(\FinSet/\SSub)^\Leftarrow$ (Definition~\ref{def:cat_of_lists}),
  that is, a  list of surjective submersions indexed by the finite set
  $X$, 
  $\Pi: ((\FinSet/\SSub)^\Leftarrow)^\op \to \SSub$ is the product
  functor (Remark~\ref{rmrk:2.41}) so that $\Pi(\tau) = \prod _{x\in
    X} \tau(x)$, 
  and $\psi$ is an interconnection morphism (Definition~\ref{def:ssub}).
\end{definition}

\begin{example} \label{ex:6.2}
Let $M_1, M_2$ be two manifolds with corners, $M_1\times M_2 \xrightarrow{\pi_1}
M_1$, $M_1\times M_2 \xrightarrow{\pi_2}
M_2$ the two associated submersions and $\varphi: (M_1\times
M_2\xrightarrow{\id}M_1\times M_2)\to (M_1\times M_2 \xrightarrow{\pi_1}
M_1)\times(M_1\times M_2 \xrightarrow{\pi_2} M_2)$ be the
interconnection map as in Example~\ref{rmrk:vf_on_product}.   This
data is 
a network of open systems.  Namely,
the indexing set $X$ is the two element set $\{1,2\}$. The map
$\tau:X\to \SSub$ is given by
\[
\tau(i) = (M_1\times M_2\xrightarrow{\pi_i}M_i),\qquad  i=1,2.
\]
The submersion $b$ is $ (M_1\times M_2 \xrightarrow{\id} M_1\times
M_2)$ and $\psi$ is the interconnection map $\varphi$ above.
\end{example}

\begin{example}\label{ex:6.3}
Fix two manifolds with corners $M$ and $U$.  As we have seen in
Example~\ref{ex:intercon} a map $s:M\to U$ defines an interconnection
map $\nu: (M\xrightarrow{\id}M)\to (U\times M\to M)$ with
$\nu_\tot (m) = (s(m), m)$.   Consider $\mu:\{*\} \to \SSub$
given by $\mu(*) = (U\times M\to M)$.   Then $(\mu, \nu:
(M\xrightarrow{\id}M)\to (U\times M\to M) = \Pi(\mu))$ is a
network of open systems.
 \end{example} 

\begin{example}\label{ex:6.4}
Let $X=\{1,2,3\}$ be a set with three elements, $M,U$ two manifolds
with corners and $s:M\to U$ a smooth map.   Consider a list $\tau:X\to
\SSub$ of surjective submersions defined by
\[
\tau(i) = (M\times U\to M) 
  \]
for $i=1,2,3$.   Let $b = (M\xrightarrow{\id}M)^3$.  We define a smooth map $\psi_\tot :M^3 \to
(M\times U)^3 = \Pi(\tau)_\tot$ by
\[
\psi_\tot (m_1, m_2, m_3) = ((m_1, s(m_2)), (m_2, s(m_1)), (m_3,
s(m_2))).
\]
Then
$\psi = (\psi_\tot, \id_{M^3}): (M\xrightarrow{\id}M)^3 \to \Pi
(\tau))$ is an interconnection map.  Hence the pair
$(\tau:X\to \SSub, \psi: b=(M\xrightarrow{\id}M)^3 \to \Pi(\tau))$ is
a network.

Note that since $\psi$ is an interconnection map, it induces a
linear map
\[
\psi^*:\Crl (\Pi(\tau)) \to \Crl
((M\xrightarrow{\id}M)^3) = \scrX (M^3)
\]
(see Definition~\ref{def:2.35}).   Consequently any three open systems
$w_1,w_2, w_3\in \Crl(M\times U\to M)$ give rise to  an open system
$w_1\times w_2 \times  w_3: (M\times U)^3 \to TM^3$.   Applying the
map $\psi^*$ we get   $v= \psi^*
(w_1\times w_2 \times  w_3) $, which is an element of $\scrX (M^3)$.
That is, $v$ a vector field on $M^3$.   It is not hard to see that
\[
  v(m_1,m_2, m_3) = %
  (w_1(m_1, s(m_2)), w_2(m_2, s(m_1)), w_3(m_3, s(m_2)))
\]
for all $(m_1, m_2, m_3)\in M^3$.     Intuitively the dynamical system
$(M^3, v)$ is made up of three interacting (open) subsystems with the
first driving the second, the second driving the first and the third.
It is in this sense that the network of open systems as defined above
is a pattern of interconnections of open systems.
We can picture the relation between the subsystems as a 
graph:
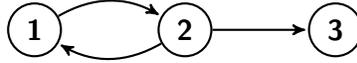
\begin{center}
\tikzset{every loop/.style={min distance=15mm,in=-30,out=30,looseness=10}}
\begin{tikzpicture}[->,>=stealth',shorten >=1pt,auto,node distance=2cm,
  thick,main node/.style={circle,draw,font=\sffamily\bfseries}]
  \node[main node] at (2,-1) (1) {1};
  \node[main node] at (4,-1) (2) {2};
  \node[main node] at (6,-1) (3) {3};
  \path[]
  (1) edge [bend left] node {{}} (2)
  (2) edge [bend left] node [below] {{}} (1)
  (2) edge  node {{}}(3)  ;
\end{tikzpicture}.
\end{center}
The arrow from node 1 to node 2 indicates that the first subsystem drives
the second.  The two arrows coming out of node 2 indicate that the
second subsystem drives the first and the third.  Thus our notion of a
network of open systems refines the common idea of a network as a
directed graph.
\end{example}    

\begin{definition}[Maps between networks of open systems]
\label{def:map_of_networks} 
A {\sf map} from a network $(X\xrightarrow{\tau} \SSub, b\xrightarrow{\psi}\Pi
(\tau) )$ of open systems to a network
$(Y\xrightarrow{\mu}\SSub, c\xrightarrow{\nu}\Pi (\mu)$
is a  pair $((\varphi, \Phi), f)$ where $(\varphi, \Phi):\tau \to \mu$
is a morphism in $(\FinSet/\SSub)^\Leftarrow$ and $f:c\to b$ is a
morphism in $\SSub$  so that the diagram %
\[
\xy
(-12, 10)*+{\Pi(\mu)} ="1"; 
(12, 10)*+{\Pi(\tau)} ="2";
(-12,-10)*+{c }="3";
(12,-10)*+{b }="4";
{\ar@{->}^{\nu} "3";"1"};
{\ar@{->}^{\Pi (\varphi, \Phi)} "1";"2"};
{\ar@{->}_{\psi} "4";"2"};
{\ar@{->}_{f} "3";"4"};
\endxy 
\]
commutes in the category $\SSub$ of surjective submersions.  That is, the diagram  is a
2-cell in the double category $\SSub^\Box$ with the source $f$ and target
$\Pi (\varphi, \Phi)$.  Here as before
$\Pi: (\FinSet/\SSub)^\Leftarrow \to \SSub$ is the product functor
(Remark~\ref{rmrk:2.41}).
\end{definition}

\begin{example} \label{ex:6.6}
Let $(\tau: \{1,2,3\}\to \SSub, \psi: (M\to M)^3\to \Pi(\tau))$ be the
network of Example~\ref{ex:6.4} and $(\mu:\{*\}\to \SSub, \nu:(M\to
M)\to \Pi(\mu))$ the network of Example~\ref{ex:6.3}.   Define
$f:(M\to M) \to (M\to M)^3$ to be the diagonal map.  Define the map of
lists $(\varphi, \Phi): \tau \to \mu$ in $(\FinSet/\SSub)^\Leftarrow$
by setting $\varphi(i) = *$ for all $i$ and defining $\Phi_i: \mu
(\varphi(i)) \to \tau(i)$ to be the identity map for all $i\in
\{1,2,3\}$.  It is not hard to check that  $((\varphi,\Phi), f):
(\tau,\psi) \to (\mu, \nu)$ is a map of networks of open systems.  It can be pictured as a map of graphs:
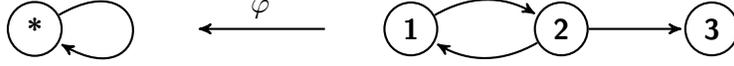
\begin{center}
\tikzset{every loop/.style={min distance=15mm,in=-30,out=30,looseness=10}}
\begin{tikzpicture}[->,>=stealth',shorten >=1pt,auto,node distance=2cm,
  thick,main node/.style={circle,draw,font=\sffamily\bfseries}]

  \node[main node] at (2,-1) (1) {1};
  \node[main node] at (4,-1) (2) {2};
  \node[main node] at (6,-1) (3) {3};
  \node at (1,-1) (rb){};
  \node at (-1,-1) (lb){};
  \node[main node] at (-3,-1) (circ){*};

  \path[]
  (1) edge [bend left] node {} (2)
  (2) edge [bend left] node [below] {} (1)
  (2) edge  node {}(3)
  (circ) edge [loop right=20] node {} (circ)
  (rb) edge node [above] {$\varphi$} (lb) ;
\end{tikzpicture}.
\end{center}

Note that given an open system $w\in \Crl (M\times U\to M)$, $u=
\nu^*w$ is a vector field on $M$ and $v= \psi^* (w\times  w \times w)$
is a vector field on $M^3$.   It is easy to check directly that the
vector field $u$ is $f$-related to the vector field $v$.   In other
words for any choice of an open system $w\in \Crl(M\times U\to M)$ we get a map of
continuous time
dynamical systems
\[
f: (M, \nu^*w) \to (M^3, \psi^* (w\times w\times w)).  
\]
Consequently for any choice of an open system $w\in \Crl(M\times U\to
M)$ the vector field $v= \psi^* (w\times  w \times w)$ on $M^3$ has
the diagonal $f(M)\subset M^3$ as an invariant subsystem.
 The fact that $f$ is a map of dynamical systems is 
special case of Theorem~\ref{thm:main_inL2} below.
\end{example}

Our definitions of networks of hybrid systems and maps of networks of
hybrid systems are analogous to the two definitions above.

\begin{definition} \label{def:network_hy_os}
  A {\sf network of hybrid open systems} is a pair $(X\xrightarrow{\tau}
  \HySSub, b\xrightarrow{\psi}  \Pi(\tau)  )$ where $\tau$ is an element
  of $(\FinSet/\HySSub)^\Leftarrow$
  (Definition~\ref{def:cat_of_lists}), i.e., a list of hybrid
  surjective submersions indexed by the finite set $X$, $\Pi:
(\FinSet/\HySSub)^\Leftarrow \to \HySSub$ is the product functor
(Remark~\ref{rmrk:2.41}), so that  $\Pi (\tau) = \prod_{x\in X} \tau(x)$,
  and $\psi$ is an interconnection morphism (Definition~\ref{def:hy_int}).
\end{definition}

\begin{definition}[Maps between networks of hybrid open systems]
\label{def:map_of_networks_hy_os} 
A {\sf map} from a network $(X\xrightarrow{\tau} \HySSub, b\xrightarrow{\psi}\Pi
(\tau))$ of hybrid open systems to a network
$(Y\xrightarrow{\mu}\HySSub, c\xrightarrow{\nu}\Pi (\mu))$
is a  pair $((\varphi, \Phi), f)$ where $(\varphi, \Phi):\tau \to \mu$
is a morphism in $(\FinSet/\HySSub)^\Leftarrow$ and $f:c\to b$ is a
morphism in $\HySSub$ so that the diagram 
\[
\xy
(-12, 10)*+{\Pi(\mu)} ="1"; 
(12, 10)*+{\Pi(\tau)} ="2";
(-12,-10)*+{c }="3";
(12,-10)*+{b }="4";
{\ar@{->}^{\nu} "3";"1"};
{\ar@{->}^{\Pi (\varphi, \Phi)} "1";"2"};
{\ar@{->}_{\psi} "4";"2"};
{\ar@{->}_{f} "3";"4"};
\endxy 
\]
it commutes
in $\HySSub$ of hybrid surjective submersions.  That is,  the diagram is a 
is a 2-cell in the double category $\HySSub^\Box$ with the source $f$
and target $\Pi (\varphi, \Phi)$. Here as before $\Pi:
(\FinSet/\HySSub)^\Leftarrow \to \HySSub$ is the product functor (Remark~\ref{rmrk:2.41}).
\end{definition}

\begin{example}\label{ex:6.9}\label{ex:6.2h}
Let $a_1,a_2$ be two hybrid phase spaces, $a_1\times a_2\xrightarrow{\pi_1}
a_i$, $i=1,2$ two associated hybrid surjective submersions and $\varphi: (a_1\times
a_2\xrightarrow{\id}a_1\times a_2)\to (a_1\times a_2\xrightarrow{\pi_1}
a_1)\times(a_1\times a_2 \xrightarrow{\pi_2} a_2)$ be the
interconnection map as in Example~\ref{ex:hds_on_product}.   This
data is 
a network of hybrid open systems in the following sense.

The indexing set $X$ is the two element set $\{1,2\}$. The map
$\tau:X\to \HySSub$ is given by
\[
\tau(i) = (a_1\times a_2\xrightarrow{\pi_i}i),\qquad  i=1,2.
\]
The hybrid submersion $b$ is $ (a_1\times a_2 \xrightarrow{\id} a_1\times
a_2)$ and $\psi$ is the interconnection map $\varphi = (\varphi_\tot,
\id)$ is defined by setting $\varphi_\tot:(a_1\times a_2)\to
(a_1\times a_2)^2$ to be the diagonal map (recall that the diagonal
maps make sense in any category with finite products).
\end{example}

\begin{example} \label{ex:6.3h}\label{ex:6.10} This example is a
  hybrid analogue of Example~\ref{ex:6.3}.   Let $m, u$ be two hybrid
  phase spaces and $s:m \to u$ a map of hybrid phase spaces.   Define a
  map $\nu: (m\xrightarrow{\id} m)\to (m\times u
  \xrightarrow{\pi_1}m)$ of hybrid surjective submersions by setting
  $\nu_\tot = (\id_{m}, s)$ and $\nu_\st =\id_{m}$.  Then
  $\nu$ is an interconnection map.    Define $\mu:\{*\}\to \HySSub$ by
  setting $\mu(*) = (m\times u\to m)$.   Then $(\mu, \nu:
  (m\xrightarrow{\id} m)\to \Pi (\mu))$ is a network of hybrid
  open systems.
  \end{example}

\begin{example} \label{ex:6.4h}\label{ex:6.11} This example is a
  hybrid analogue of Example~\ref{ex:6.4}.

  Let $X=\{1,2,3\}$ be a set with three elements, $m,u$ two hybrid
  phase spaces and $s:m\to u$ a map of hybrid phase spaces.   Consider a list $\tau:X\to
\HySSub$ of hybrid surjective submersions defined by
\[
\tau(i) = (m\times u\to m) 
  \]
for $i=1,2,3$.   Then $\Pi(\tau) =(m\times u\to m)^3 = (m^3 \times u^3
\to m^3)$. Let $b = (m\xrightarrow{\id}m)^3$.   We define an
interconnection map $\psi: b\to \Pi(\tau)$ by making use of the
functor $\Pi:((\FinSet/\HyPh)^\Leftarrow)^\op \to \HyPh$ (see
Remark~\ref{rmrk:2.41} and Example~\ref{ex:exPiphi2}).

Namely consider the map $\varphi:\{1,2,3\}\to \{1,2,3\}$ which is defined by 
  \[
\varphi(1) = 2,\qquad \varphi(2) = 1, \qquad \varphi(3) =2.
\]
Define $\alpha, \beta:\{1,2,3\}\to \HyPh$ by setting
\[
\alpha (i) = u, \qquad \beta (j) = m\qquad \textrm{for all }i,j\in \{1,2,3\}.
\]
Define $(\varphi, \Phi): \alpha \to \beta$ in
$(\FinSet/\HyPh)^\Leftarrow$ by setting $\Phi_i: \beta (\varphi(i))\to
\alpha(i)$ be the map $s:m\to u$ for all $i$.    Now define
$\psi_\tot : m^3 = b_\tot \to \Pi(\tau)_\tot = m^3 \times u^3$ by
setting
\[
\psi_\tot = (\id_{m^3}, \Pi (\varphi, \Phi)).
\]
Then  $\psi$ is an interconnection map of hybrid surjective
submersions and $(\tau:\{1,2,3\}\to \HySSub, \psi:
b=(m\xrightarrow{\id}m)^3 \to \Pi(\tau))$ is a network of hybrid open systems.
Note that the map $\UU(\psi_\tot): \UU(m)^3 \to \UU(m)^3 \times \UU
(u)^3$ of manifolds with corners is given by
\[
(x_1, x_2, x_3) \mapsto (x_1, x_2, x_3, s(x_2), s(x_1), s(x_3)).
\]
This follows from Corollary~\ref{cor:4.12} and Example~\ref{ex:exPiphi2}.
\end{example}

\begin{example} \label{ex:6.6h} \label{ex:6.12} This example is the
  hybrid analogue of Example~\ref{ex:6.6}.

 Let $(\mu, \nu:
  (m\xrightarrow{\id} m)\to \Pi (\mu))$ be the network of hybrid  open
  systems of Example~\ref{ex:6.3h}.   Let $(\tau:\{1,2,3\}\to \HySSub, \psi:
b=(m\xrightarrow{\id}m)^3 \to \Pi(\tau))$ be the network of hybrid
open systems of Example~\ref{ex:6.4h}.  Define $f:(m\to m)\to (m\to
m)^3$ to be the diagonal map.  Define the map of
lists $(\varphi, \Phi): \tau \to \mu$ in $(\FinSet/\HySSub)^\Leftarrow$
by setting $\varphi(i) = *$ for all $i$ and defining $\Phi_i: \mu
(\varphi(i)) \to \tau(i)$ to be the identity map for all $i\in
\{1,2,3\}$.   The pair
$((\varphi,\Phi), f)$ is a map of networks of hybrid open systems from
$(\tau, \psi)$ to $(\mu, \nu)$.   This map of networks can also be
pictured as a map of directed graphs

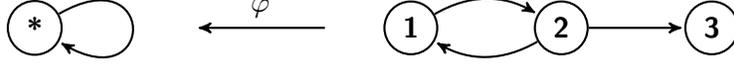
\begin{center}
\tikzset{every loop/.style={min distance=15mm,in=-30,out=30,looseness=10}}
\begin{tikzpicture}[->,>=stealth',shorten >=1pt,auto,node distance=2cm,
  thick,main node/.style={circle,draw,font=\sffamily\bfseries}]

  \node[main node] at (2,-1) (1) {1};
  \node[main node] at (4,-1) (2) {2};
  \node[main node] at (6,-1) (3) {3};
  \node at (1,-1) (rb){};
  \node at (-1,-1) (lb){};
  \node[main node] at (-3,-1) (circ){*};

  \path[]
  (1) edge [bend left] node {} (2)
  (2) edge [bend left] node [below] {} (1)
  (2) edge  node {}(3)
  (circ) edge [loop right=20] node {} (circ)
  (rb) edge node [above] {$\varphi$} (lb) ;
\end{tikzpicture}.
\end{center}
Note that now we have hybrid phase spaces attached to the nodes of the graphs.
\end{example}
\mbox{}

To state the main result of \cite{L2} and then to state and prove its
generalization to networks of hybrid open systems we need to recall a
few facts from \cite{L2}.  Note first that the direct sum is not a
coproduct in the category $\RelVect$ of vector spaces and linear
relations.  None the less, an analogue of Remark~\ref{rmrk:2.43}
holds.

\begin{proposition}\label{prop:L24.2}
The assignment 
\[
(\FinSet/\RelVect) \ni \tau \mapsto \odot (\tau):= \oplus_{a\in X} \tau(a) \in \RelVect
\]
extends to a lax  functor
\[
\odot: ((\FinSet/\RelVect)^\Leftarrow)^{\op}   \to \RelVect.
\]
\end{proposition}
\begin{proof} Given a 2-commuting triangle 
\[
\xy
(-10, 10)*+{X} ="1"; 
(10, 10)*+{Y} ="2";
(0,-2)*+{\RelVect }="3";
{\ar@{->}_{\tau} "1";"3"};
{\ar@{->}^{\varphi} "1";"2"};
{\ar@{->}^{\mu} "2";"3"};
{\ar@{=>}_{\scriptstyle \Phi} (4,6)*{};(-0.4,4)*{}} ; 
\endxy 
\]
we set
\begin{equation}\label{eq:5*}
\odot (\varphi, \Phi):= \bigcap _{a\in X} (\pi_{\varphi(a)} \times
\pi_{a})\inv (\Phi_a),
\end{equation}
where the relations $\Phi_a:\mu(\varphi(a))\to\tau (a)$ are
the component relations of the natural transformation $\Phi: \mu\circ
\varphi \Rightarrow \tau$ and 
\[
\pi_{\varphi (a)}\times \pi_{a}: \oplus_{y\in Y}\mu(y) \times \oplus
_{x\in X} \tau(x) \to \mu(\varphi((a))\times \mu(\tau(a)) 
\]
are the projections. It remains to show that $\odot$ is actually a
functor.  We refer the reader to the proof of
\cite[Proposition~5.8]{L2} for such an argument.  Note that \cite{L2}
unconventionally views a linear relation
$\Phi_a:\mu (\varphi(a)) \to \tau(a)$ as a subspace of
$\tau(a)\times \mu(\varphi(a))$ and not of
$\mu(\varphi((a))\times \mu(\tau(a))$.
 \end{proof}

Next we recall that the control functor $\Crl:\SSub \to \RelVect$ is monoidal,
where the monoidal structure on the category $\SSub$ is the
categorical product and on $\RelVect$ is the direct sum.  In
particular, for any two surjective submersions $a,b$ we have a linear map
\[
\Crl_{a,b}: \Crl(a)\oplus \Crl(b) \to \Crl(a\times b)
\]
which is given by
\[
\Crl_{a,b}(F,G) := F\times G
\]
for all $(F,G)\in \Crl(a)\oplus \Crl(b)$.  It follows that for any 
{\bf ordered} list
$a:\{1,\ldots, n\}\to \SSub$ of submersions we have a canonical linear
map
\[
\Crl_a: \bigoplus_{i=1}^n \Crl(a(i))\to \Crl (\prod _{i=1}^n a(i)).
\]
It is given by 
\begin{equation}\label{eq:L25.2}
(\psi_1,\ldots, \psi_n)\mapsto \psi_1\times \cdots \times \psi_n.
\end{equation}

\begin{lemma}\label{lem:L25.14} 
For any unordered  list $\tau: X\to \SSub$ we have a canonical
linear map
\begin{equation} \label{eq:L25.16}
 \Crl_\tau: \bigoplus _{x\in X}\Crl(\tau(x))\to \Crl(\prod_{x\in X} \tau(x)) 
\end{equation}
so that if  $X=\{1,\ldots,n\}$ then $\Crl_\tau$ is given by  \eqref{eq:L25.2}.
\end{lemma}
\begin{proof}
See \cite[Lemma~5.15]{L2}.
  \end{proof}
 
We can now state the main result of \cite{L2}.
\begin{theorem}\label{thm:main_inL2}
  A map 
\[
((\varphi,\Phi), f):(\tau:X\to \SSub,
  \psi:b\to\Pi(\tau)) \to (\mu:Y\to
  \SSub, \nu:c\to\Pi (\mu))
\]
of networks of open systems 
gives rise to a 2-cell
\[
 \xy
 (-25, 10)*+{\bigoplus_{y\in Y}\Crl(\mu(y) )} ="1"; 
(25, 10)*+{\bigoplus_{x\in X} \Crl(\tau(x))}="2";
 (-25,-10)*+{\Crl(c)}="3";
(25,-10)*+{\Crl(b) }="4";
{\ar@{->}_ {\nu^*\circ \Crl_\mu}  "1";"3"};
 {\ar@{->}^{\odot  (\varphi, \Crl\circ \Phi)} "1";"2"};
 {\ar@{<-}_{\psi^*\circ \Crl_\tau} "4";"2"};
 {\ar@{->}_{\Crl(f)} "3";"4"};
{\ar@{<=} (0,-2)*{};(0,2)*{}} ; 
 \endxy 
\]
 in the double category $\RelVect^\Box$ of vector spaces, linear maps and
linear relations. (The functor $\odot$ is defined in Proposition~\ref{prop:L24.2},
 the maps $\Crl_\mu, \Crl_\tau$ come from Lemma~\ref{lem:L25.14},
and the pullback maps $\nu^*, \psi^*$ are from Definition~\ref{def:2.35}.)
\end{theorem}
\mbox{}

We are now in position to state and prove  the main result of the paper
\begin{theorem}[Main theorem]\label{thm:main_result}
  A map
  $\big((\varphi,\Phi),f\big):
  (\tau:X\rightarrow\HySSub,\psi:b\rightarrow\Pi(\tau)
  )\rightarrow(\mu:Y\rightarrow\HySSub,\nu:c\rightarrow\Pi(\mu) )$ of
  networks of open systems gives rise to a 2-cell
\[
\xy
(-35, 10)*+{\oplus_{y\in Y}\Crl ((\UU ( \mu(y)))} ="1"; 
(35, 10)*+{\oplus_{x\in X}\Crl(\UU (\tau(x) ))} ="2";
(-35,-10)*+{\Crl(\UU(c)) }="3";
 (35,-10)*+{\Crl(\UU(b)) }="4";
{\ar@{->}_{(\UU(\nu))^* \circ \Crl_{\UU\circ \mu} }"1";"3"};
{\ar@{->}^{\odot  (\varphi, \Crl \circ \UU \circ \Phi))} "1";"2"};
 {\ar@{->}^{(\UU(\psi))^* \circ \Crl_{\UU\circ \tau}} "2";"4"};
 {\ar@{->}_{\Crl(\UU(f))} "3";"4"};
 {\ar@{<=} (0,-2)*{};(0,2)*{}} ; 
\endxy 
\]   
in the double category $\RelVect^\Box$ of vector spaces, linear maps
and relations.  (The functor $\odot$ is defined in Proposition~\ref{prop:L24.2},
 the maps $\Crl_{\UU\circ \mu}, \Crl_{\UU \circ \tau}$ come from
 Lemma~\ref{lem:L25.14}, 
and the pullback maps $(\UU(\nu)^*, \UU(\psi)^*$ are from Definition~\ref{def:2.35}.)
\end{theorem}

\begin{remark}
Theorem~\ref{thm:main_result}  asserts that for any choice $\{w_x\in
\Crl(\UU(\tau(x)))\}_{x\in X}$, $\{u_y\in \Crl (\UU(\mu(y)))\}_{y\in Y}$ of
open  systems so that the open system  $u_{\varphi(x)}$ is
$\UU(\Phi_x)$-related to the open system $w_x$ for all $x\in X$ we get a map
\[
f: \left(c, (\UU(\nu)^*\circ \Crl_{\UU\circ \mu}) ((u_y)_{y\in Y}) \right)\to \left(b,
(\UU(\psi)^* \circ \Crl_{\UU \circ \tau})((w_x)_{x\in X})\right)
  \]
of hybrid open systems.   In other words, compatible patterns of
interconnection of hybrid open systems give rise to maps of hybrid open
systems.  In the case where the two interconnections result in closed
systems, the map  $f$ a map of hybrid dynamical systems.
 \end{remark} 
\begin{proof}[Proof of Theorem~\ref{thm:main_result}] We apply the
  product-preserving functor $\UU:\HySSub\to \SSub$ to the data of the
  theorem.  Namely, by applying the functor $\UU$ to the 2-commuting triangle
\[
\xy
(-10, 10)*+{X} ="1"; 
(10, 10)*+{Y} ="2";
(0,-2)*+{\HySSub}="3";
{\ar@{->}_{\tau} "1";"3"};
{\ar@{->}^{\varphi} "1";"2"};
{\ar@{->}^{\mu} "2";"3"};
{\ar@{=>}_{\scriptstyle \Phi} (4,6)*{};(-0.4,4)*{}} ; 
\endxy 
\]
we get
\begin{equation}\label{eq:5.14}
\xy
(-10, 10)*+{X} ="1"; 
(10, 10)*+{Y} ="2";
(0,-2)*+{\SSub }="3";
{\ar@{->}_{\UU\circ \tau} "1";"3"};
{\ar@{->}^{\varphi} "1";"2"};
{\ar@{->}^{\UU \circ \mu} "2";"3"};
{\ar@{=>}_{\scriptstyle \UU\circ \Phi} (4,6)*{};(-0.4,4)*{}} ; 
\endxy .
\end{equation}
Applying $\UU$ to the 2-cell in the double category $\HySSub^\Box$
\[
\xy
(-12, 10)*+{\Pi(\mu)} ="1"; 
(12, 10)*+{\Pi(\tau)} ="2";
(-12,-10)*+{c }="3";
(12,-10)*+{b }="4";
{\ar@{->}^{\nu} "3";"1"};
{\ar@{->}^{\Pi (\varphi, \Phi)} "1";"2"};
{\ar@{->}_{\psi} "4";"2"};
{\ar@{->}_{f} "3";"4"};
{\ar@{=>} (0,-2)*{};(0,2)*{}} ; 
\endxy ,
\]
which is a commuting diagram in $\HySSub$, 
we get a commuting diagram
\begin{equation} \label{eq:5.15}
\xy
(-20, 10)*+{\Pi(\UU \circ \mu)} ="1"; 
(20, 10)*+{\Pi(\UU \circ \tau)} ="2";
(-20,-10)*+{\UU(c) }="3";
(20,-10)*+{\UU(b) }="4";
{\ar@{->}^{\UU(\nu)} "3";"1"};
{\ar@{->}^{\UU(\Pi (\varphi, \Phi))} "1";"2"};
{\ar@{->}_{\UU(\psi)} "4";"2"};
{\ar@{->}_{\UU(f)} "3";"4"};
\endxy 
\end{equation}
is $\SSub$.  Recall that a map of hybrid surjective submersions is an
interconnection map if and only if its image under $\UU$ is an interconnection map
between the corresponding surjective submersions.  It follows that
\eqref{eq:5.15} is  a 2-cell in the double category $\SSub^\Box$.

Since $\UU$ is product preserving we may assume that that $\UU (\Pi
(\tau)) = \Pi (\UU\circ \tau)$ and similarly for $\mu$.   Similarly,
since $\Pi (\varphi, \UU \circ \Phi): \Pi (\UU\circ \mu)\to \Pi (\UU
\circ \tau)$ is defined by the universal properties, 
\[
\Pi (\varphi, \UU\circ \Phi) = \UU (\Pi (\varphi, \Phi))
  \]
once we identify $\UU (\Pi
(\tau))$ with  $\Pi (\UU\circ \tau)$ and $\UU (\Pi (\mu)) $ with $\Pi
(\UU \circ \mu))$.  After these identifications the 2-cell
\eqref{eq:5.15} is
\begin{equation} \label{eq:5.16}
\xy
(-20, 10)*+{\Pi(\UU \circ \mu)} ="1"; 
(20, 10)*+{\Pi(\UU \circ \tau)} ="2";
(-20,-10)*+{\UU(c) }="3";
(20,-10)*+{\UU(b) }="4";
{\ar@{->}^{\UU(\nu)} "3";"1"};
{\ar@{->}^{\Pi (\varphi, \UU \circ \Phi))} "1";"2"};
{\ar@{->}_{\UU(\psi)} "4";"2"};
{\ar@{->}_{\UU(f)} "3";"4"};
{\ar@{=>} (0,-2)*{};(0,2)*{}} ; 
\endxy 
\end{equation}
The diagrams \eqref{eq:5.14} and \eqref{eq:5.16} together define a map
\[
((\varphi, \UU \circ \Phi), \UU(f)): (X\xrightarrow{\UU\circ
  \tau}\SSub, \UU(b)\xrightarrow{\UU(\psi)} \Pi (\UU\circ \tau)) \to
(Y\xrightarrow{\UU\circ \mu} \SSub, \UU(c)\xrightarrow{\UU(\nu)}\Pi
(\UU\circ \mu)))
  \]
of networks of open systems.   By Theorem~\ref{thm:main_inL2} we have
the 2-cell 
\[
\xy
(-35, 10)*+{\oplus_{y\in Y}\Crl ((\UU ( \mu(y)))} ="1"; 
(35, 10)*+{\oplus_{x\in X}\Crl(\UU (\tau(x) ))} ="2";
(-35,-10)*+{\Crl(\UU(c)) }="3";
 (35,-10)*+{\Crl(\UU(b)) }="4";
{\ar@{->}_{(\UU(\nu))^* \circ \Crl_{\UU\circ \mu} }"1";"3"};
{\ar@{->}^{\odot  (\varphi, \Crl \circ \UU \circ \Phi))} "1";"2"};
 {\ar@{->}^{(\UU(\psi))^* \circ \Crl_{\UU\circ \tau}} "2";"4"};
 {\ar@{->}_{\Crl(\UU(f))} "3";"4"};
 {\ar@{<=} (0,-2)*{};(0,2)*{}} ; 
\endxy 
\]
in the double category $\RelVect^\Box$ which is what we wanted to prove.
 \end{proof}     

\appendix
\section{Manifolds with corners} \label{app:A}

We recall the definition of a manifold with corners which is fairly
standard (see, for example, \cite{Michor}).

\begin{definition}[Manifold with corners] \label{md w corners}
Let $V$ be an arbitrary subset of $\R^n$.  A map $\varphi :V \to \R^m$
is {\sf smooth} if for every point $p$ of $V$ there exist
an open subset $\Omega$ in $\R^n$ containing $p$ and a smooth map
from $\Omega$ to $\R^m$ whose restriction to $\Omega \cap V$
coincides with $\varphi|_{\Omega \cap V}$, i.e., $\varphi$ extends
locally to a smooth map in the conventional sense.
A map $\varphi$ from the set $V$ to a {\em subset} of $\R^m$ is {\sf smooth}
if it is smooth as a map from $V$ to $\R^m$.  A map $\varphi$ from a subset $V$
of $\R^n$ to a subset $W$ of $\R^m$ is a {\sf diffeomorphism}
if $\varphi$ is a bijection and both $\varphi$  and its inverse
$\varphi\inv:W\to V$ are smooth. 

A {\sf sector} is the space $[0,\infty)^k \times \R^{n-k}$ where $n$
is a nonnegative integer and $k$ is an integer between $0$ and $n$.
We equip the sector with the topology inherited from its inclusion in
$\R^n$.  Let $X$ be a Hausdorff paracompact topological space.  A
{\sf chart} on an open subset $U$ of $X$ is a homeomorphism $\varphi$
from $U$ to an open subset $V$ of a sector.  Charts
$\varphi : U \to V$ and $\varphi' : U' \to V'$ are {\sf compatible} if
$\varphi' \circ \varphi\inv$ is a diffeomorphism from
$\varphi(U \cap U')$ to $\varphi'(U \cap U')$.  An {\sf atlas} on a
Hausdorff paracompact topological space $X$ is a set of pairwise
compatible charts whose domains cover $X$. Two atlases are
{\sf equivalent} if their union is an atlas.  A {\sf manifold with
  corners} is a Hausdorff paracompact topological space equipped
with an equivalence class of atlases.
\end{definition}

There are many incompatible notions of smooth maps between manifolds
with corners.   We use the following definition.
\begin{definition}[Smooth map]
Let $M$ and $M'$ be manifolds with corners.
A map $h:M\to M'$ is {\sf smooth}
if for every point
$p$ in $M$ there exists an open neighborhood $U$ of $p$ in  $M$
and an open neighborhood  $U'$ of  $h(p) $ in $M'$
and charts $\varphi:U \to V$ and $\varphi': U' \to V'$ 
of $M$ and $M'$ respectively 
such that $h(U) \subset U'$
and such that $\varphi' \circ h \circ \varphi\inv :V \to V'$ is smooth
in the sense of Definition~\ref{md w corners}.
\end{definition}

\begin{notation}[The category $\Man$ of manifolds with corners]\label{category:ManifoldsWithCorners}
  Manifolds with corners and smooth maps between them form an evident
  category.  We denote it by $\Man$.
\end{notation}

\begin{definition}\label{differential forms}
The {\sf tangent space} $T_xM$ of a manifold with corners $M$ at a
point $x\in M$ is the space of $\R$-valued derivations at $x$ of germs at $x$ of
smooth functions defined near $x$.
\end{definition}

Thus, the tangent space at a point $x$ of a manifold with corners $M$ is a
vector space even if the point $x$ is in the topological boundary of $M$.
Similarly the tangent bundle $TM$ of a manifold with corners $M$ is a
vector bundle over $M$.  The total spaces of $TM$ is a manifold with
corners and the canonical projection $\pi:TM\to M$ is smooth (c.f.\
\cite[p.\ 19]{Michor}).

\begin{definition}
  A {\sf vector field $X$ on a manifold with corners $M$} is a section
  of its tangent bundle $TM \to M$.  An {\sf integral curve}
  of a  vector field $X$ is a smooth map $x:I\to M$ where $I\subset
  \R$ is an
  interval, which may be open or closed, bounded or unbounded.  We
  require that 
\[
\frac{d}{dt}x =
X(x(t))
\]
for all $t\in I$.
\end{definition}

\begin{notation}
We denote the vector space of vector fields on a manifold with corners $M$ by 
$\mathscr{X}(M)$.
\end{notation}

\section{A traditional  definition of a hybrid dynamical system}
\label{app:B}
There is a variety of definitions of hybrid dynamical systems in
literature.  In this appendix we choose to present the definition that
generalizes the notion of a continuous time dynamical system.  Other
definitions, for example,  generalize labelled transition systems.
All of these definitions involve the notion of a directed
graph, which we presently recall to fix our notation.

\begin{definition}
For the purposes of this paper a {\sf graph} $A$ is a directed
multigraph.  Thus $A$ is a pair of collections $A_0$ (nodes, vertices) and $A_1$ (arrows,
  edges) together with two maps $\frs,\frt:A_1\to A_0$ (source and
  target).  We do not require that $A_1, A_0$ are sets in the sense of ZFC.

We depict an arrow $\gamma \in A_1$ with the source $a$ and target $b$
as $a\xrightarrow{\gamma}b$.  We write $A=\{A_1\toto A_0\}$ to
remind ourselves that our graph $A$ consists of two collections  and two maps.
\end{definition}

\begin{remark}
  Every category has an underlying graph: forget the composition of
  morphisms.  The the collections of objects and morphisms in a
  given category may be too big to be sets of ZFC.  Consequently the
  collections of vertices and edges in the corresponding underlying
  graph are  also too  big to be sets.   This causes no problems.
\end{remark}

The following definition of a hybrid dynamical system is a slight
variant of
 \cite[Definition~2.1]{SJLS}).  Note that in \cite{SJLS} what we call
 {\em manifolds with corners} are called {\em domains}. 
 Since in mathematics and computer science
  literature the word ``domain'' has other meanings we prefer not to
  use this term.
 Another name in hybrid literature for a ``domain'' is an
 {\em invariant}.  But in mathematics  ``invariant'' has too many
 other meanings (e.g., invariant submanifolds, invariant functions,
 invariant vectors etc.) so we prefer not to use this term either.

\begin{definition}[Hybrid dynamical system]    \label{def:hds_trad}
A {\sf hybrid dynamical system} (HDS) consists of 
\begin{enumerate}
\item A  graph $A=\{A_1\toto A_0\}$;
\item For each node $x\in A_0$ a dynamical system  $(R_x, X_x)$ where $X_x$ is a vector field on the manifold with corners $R_x$.
\item For each arrow $x\xrightarrow{\gamma}y$ of $A$  a {\sf reset } relation
  $R_\gamma: R_x \to R_y$ (i.e., $R_\gamma $ is a subset of the
  product $ R_x\times R_y$).
\end{enumerate}
Thus a hybrid dynamical system is a tuple $(A =\{A_1\toto A_0\},
\{(R_x, X_x)\}_{x\in A_0}, \{R_\gamma \}_{\gamma \in A_1}))$.
\end{definition}

\begin{remark}
A common variant of the definition of the hybrid dynamical system
insists that the relations  $R_\gamma$ are partial maps
whose domains are smooth submanifolds and the partial maps themselves
are smooth.
\end{remark}

We end the appendix with a fairly standard definition of an execution
of a hybrid dynamical system.  Executions are hybrid analogues of
integral curves of vector fields.

\begin{definition}[An execution with jump times indexed 
by the natural numbers $\N$] \label{def:chs_execution}\mbox{}\\
  Let $H= (A =\{A_1\toto A_0\}, \{(R_x, X_x)\}_{x\in A_0}, \{R_\gamma
  \}_{\gamma \in A_1}))$ be a hybrid dynamical system.  An {\sf
    execution} of $H$ is
\begin{enumerate}
\item an nondecreasing sequence $\{t_i\}_{i\geq 0}$ of real numbers
\item a function $\varphi_0:\N\to A_0$;
\item a function $\varphi_1:\N\to A_1$ compatible with $\varphi_0$: we
  require that $\frs(\varphi_1(i)) = \varphi_0 (i)$ and
  $\frt(\varphi_1 (i)) = \varphi_0 (i+1)$ (where as before $\frs,
  \frt:A_1\to A_0$ are the source and target maps, respectively);
\item an integral curve $\sigma_i: [t_{i-1}, t_{i}] \to R_{\varphi_0
    (i)}$ of $X_{\varphi_0(i)}$ 
(with $t_{-1}$ being some number less than $t_0$);
\item the terminal end point of $\sigma_i$ and the initial
  end point of $\sigma_{i+1}$ are related by the reset relation
  $R_{\varphi_1(i)}$: 
\[
(\sigma_i (t_{i}), \sigma_{i+1} (t_{i})) \in R_{\varphi_1(i)}.
\]
\end{enumerate}
\end{definition}

 \begin{remark}
 Definition~\ref{def:chs_execution} allows for executions
that make infinitely many jumps in finite time.  This is important
for modeling  the so called Zeno behavior.   
 \end{remark}

\begin{remark}
  The only property of the natural numbers $\N$ used in
  Definition~\ref{def:chs_execution} is that $\N$ is an ordered
  countable set.  One can use any other countable ordered set $X$ to
  define executions with jump times indexed by $X$. Popular choices 
  in literature of such indexing sets include $[n] := \{0<1<\ldots <
  n\}$, where $n$ is a  positive integer, and the set $\Z$ of all integers.
 \end{remark} 

\mbox{}\\[2cm]

\end{document}